\documentclass[reqno]{amsart}
\usepackage{amsmath}
\usepackage{amsthm}
\usepackage{graphicx}
\usepackage{amsfonts}
\usepackage{amssymb}
\usepackage{enumerate}

\numberwithin{equation}{section}

\newtheorem{theorem}{Theorem}[section]
\newtheorem{lemma}[theorem]{Lemma}
\newtheorem{proposition}[theorem]{Proposition}
\newtheorem{corollary}[theorem]{Corollary}

\theoremstyle{definition}
\newtheorem{example}[theorem]{Example}
\newtheorem{definition}[theorem]{Definition}

\newtheorem*{A1}{Assumption (S)}
\newtheorem*{A2}{Assumption (E)}
\newtheorem*{A3}{Assumption (F)}
\newtheorem*{A41}{Assumption (C)}
\newtheorem*{A42}{Assumption (U)}

\newtheorem{remark}[theorem]{Remark}

\def\E{{\mathbb E}}
\def\R{{\mathbb R}}
\def\N{{\mathbb N}}
\def\P{{\mathcal P}}
\def\M{{\mathcal M}}

\def\B{{\mathcal B}}
\def\W{{\mathcal W}}
\def\V{{\mathcal V}}
\def\S{{\mathcal S}}
\def\Y{{\mathcal Y}}
\def\L{{\mathcal L}}
\def\X{{\mathcal X}}
\def\H{{\mathcal H}}
\def\Q{{\mathcal Q}}

\def\A{{\mathbb A}}
\def\F{{\mathcal F}}
\def\C{{\mathcal C}}

\title[Probabilistic weak formulation of mean field games]{A probabilistic weak formulation of mean field games and applications}
\author{Ren\'{e} Carmona}
\address{ORFE, Bendheim Center for Finance, Princeton University, Princeton, NJ  08544, USA.}
\email{rcarmona@princeton.edu}
\author{Daniel Lacker}
\address{ORFE,  Princeton University, Princeton, NJ  08544, USA.}
\email{dlacker@princeton.edu}
\thanks{Partially supported  by NSF: DMS-0806591}



\begin{document}
\maketitle

\begin{abstract}
Mean field games are studied by means of the weak formulation of stochastic optimal control. This approach allows the mean field interactions to enter through both state and control processes and take a form which is general enough to include rank and nearest-neighbor effects.  Moreover, the data may depend discontinuously on the state variable, and more generally its entire history. Existence and uniqueness results are proven, along with a procedure for identifying and constructing distributed strategies which provide approximate Nash equlibria for finite-player games. Our results are applied to a new class of multi-agent price impact models and a class of flocking models for which we prove existence of equilibria.
\end{abstract}

\section{Introduction}
The methodology of mean field games initiated by Lasry and Lions \cite{lasrylionsmfg} has provided an elegant and tractable way to study approximate Nash equilibria for large-population stochastic differential games with a so-called mean field interaction. In such games, the players' private state processes are coupled only through their empirical distribution. Borrowing intuition from statistical physics, Lasry and Lions study the system which should arise in the limit as the number of players tends to infinity. A set of strategies for the finite-player game is then derived from the solution of this limiting problem. These strategies form an approximate Nash equilibrium for the $n$-player game if $n$ is large, in the sense that no player can improve his expected reward by more than $\epsilon_n$ by unilaterally changing his strategy, where $\epsilon_n \rightarrow 0$ as $n \rightarrow \infty$ (see \cite{huangmfg1}). An attractive feature of these strategies is that they are distributed, in the sense that the strategy of a single player depends only on his own private state.

Mean field games have seen a wide variety of applications, including models of oil production, volatility formation, population dynamics, and economic growth (see \cite{lasrylionsmfg, gueantlasrylionsmfg, gueantlasrylionsmfg-growth, lachapellecrowds} for some examples). Independently, Huang, Malham\'e, and Caines developed a similar research program under the name of Nash Certainty Equivalent. The interested reader is referred to \cite{huangmfg1} and \cite{huangmfg2} and the papers cited therein. They have since generalized the framework, allowing for several different types of players and one major player. 

The finite-player games studied in this paper are summarized as follows. For $i=1,\ldots,n$, the dynamics of player $i$'s private state process are given by a stochastic differential equation (SDE):
\begin{align}
dX^i_t = b(t,X^i,\mu^n,\alpha^i_t)dt + \sigma(t,X^i)dW^i_t, \ \ X^i_0 = \xi^i, \label{introSDE}
\end{align}
where $\mu^n$ is the empirical distribution of the states:
\begin{equation}
\label{fo:empirical}
\mu^n = \frac{1}{n}\sum_{j=1}^n\delta_{X^j}.
\end{equation}
The drift $b$ may depend on time, player $i$'s private state (possibly its history), the distribution of the private states (possibly their histories), and player $i$'s own choice of control $\alpha^i_t$. Here, $W^i$ are independent Wiener processes and $\xi^i$ are independent identically distributed random variables independent of the Wiener processes, and each player has the same drift and volatility coefficients. Moreover, each player $i$ has the same objective, which is to maximize
\[
\E\left[\int_0^Tf(t,X^i,\mu^n,q^n_t,\alpha^i_t)dt + g(X^i,\mu^n)\right], \text{ where } q^n_t = \frac{1}{n}\sum_{j=1}^n\delta_{\alpha_t^j}
\]
over all admissible choices of $\alpha^i$, subject to the constraint (\ref{introSDE}). Note that the running reward function $f$ may depend upon the empirical distribution of the controls at time $t$, in addition to the same arguments as $b$. This is part of the thrust of the paper. Of course, each player's objective depends on the actions of the other players, and so we look for Nash equilibria.

Intuitively, if $n$ is large, because of the symmetry of the model, player $i$'s contribution to $\mu^n$ is negligible, and he may as well treat $\mu^n$ as fixed. This line of argument leads to the derivation of the mean field game problem, which has the following structure:
\begin{enumerate}
\item Fix a probability measure $\mu$ on path space and a flow $\nu: t \mapsto \nu_t$ of measures on the control space;
\item With $\mu$ and $\nu$ frozen, solve the standard optimal control problem:
\begin{align}
\begin{cases}
&\sup_\alpha\E\left[\int_0^Tf(t,X,\mu,\nu_t,\alpha_t)dt + g(X,\mu)\right], \text{ s.t.} \\
&dX_t = b(t,X,\mu,\alpha_t)dt + \sigma(t,X)dW_t, \ \ X_0 = \xi ;
\end{cases} \label{mfg1}
\end{align}
\item Find an optimal control $\alpha$, inject it into the dynamics of (\ref{mfg1}), and find the law $\Phi_x(\mu,\nu)$ of the optimally controlled state process, and the flow $\Phi_\alpha(\mu,\nu)$ of marginal laws  of the optimal control process;
\item Find a fixed point $\mu = \Phi_x(\mu,\nu)$, $\nu = \Phi_\alpha(\mu,\nu)$.
\end{enumerate}
This is to be interpreted as the optimization problem faced by a single representative player in a game consisting of infinitely many independent and identically distributed (i.i.d.) players. In the first three steps, the representative player determines his best response to the other players' states and controls which he treats as given. The final step is an equilibrium condition; if each player takes this approach, and there is to be any consistency, then there should be a fixed point.
Once existence and perhaps uniqueness of a fixed point are established, the second problem is to use this fixed point to construct approximate Nash equilibrium strategies for the original finite-player game. These strategies will be constructed from the optimal control for the problem of step (2), corresponding to the choosing $(\mu,\nu)$ to be the fixed point in step (1). 

The literature on mean field games comprises two streams of papers: one based on analytic methods and one on a probabilistic approach.

Lasry and Lions (e.g. \cite{lasrylionsmfg}, \cite{gueantlasrylionsmfg}, etc.) study these problems via a system of partial differential equations (PDEs). The control problem gives rise to a Hamilton-Jacobi-Bellman equation for the value function, which evolves backward in time. The law of the state process is described by a Kolmogorov equation, which evolves forward in time. These equations are coupled through the dependence on the law of the state process, in light of the consistency requirement (4). This approach applies in the Markovian case, when the data $b$, $\sigma$, $f$, and $g$ are smooth or at least continuous functions of the states and not of their pasts. Results in this direction include two broad classes of mean field interactions: some have considered local dependence of the data on the measure argument, such as functions $(x,\mu) \mapsto G(d\mu(x)/dx)$ of the density, while others have studied nonlocal functionals, which are continuous with respect to a weak or Wasserstein topology. 

More recently, several authors have taken a probabilistic approach to this problem by using the Pontryagin maximum principle to solve the optimal control problem. See, for example, \cite{carmonadelarue-mkv,bensoussan-lqmfg,carmonadelarue-mfg}. Typically in a stochastic optimal control problem, the backward stochastic differential equations (BSDEs) satisfied by the adjoint processes are coupled with the forward SDE for the state process through the optimal control, which is generally a function of both the forward and backward parts. When the maximum principle is applied to mean field games, the forward and backward equations are coupled additionally through the law of the forward part. Carmona and Delarue investigate this new type of forward-backward stochastic differential equations (FBSDEs) in \cite{carmonadelarue-mffbsde}. It should be noted that there is a similar but distinct way to analyze the infinite-player limit of large-population games, leading to the optimal control of stochastic dynamics of McKean-Vlasov type. Early forms of a stochastic maximum principle for this new type of control problem were given in \cite{and-djehiche-maximum,buckdahndjehicheli-general,oksendal-meanfield}.
A general form of this principle was given in \cite{carmonadelarue-mkv} where it was applied to the solution of the control problem. A comparison of these two asymptotic regimes is given in \cite{carmonadelarue-mkvvsmfg}.

The aim of this paper is to present a new probabilistic approach to the analysis of mean field games with uncontrolled diffusion coefficients. Assuming $\sigma = \sigma(t,x)$ contains neither a mean field term nor a control, we obtain a general existence result. Under stronger assumptions, we prove a modest extension of the uniqueness result of Lasry and Lions \cite{lasrylionsmfg}. Finally, we provide a construction of approximate Nash equilibria for finite-player games in the spirit of \cite{carmonadelarue-mfg}, in the case that $b$ has no mean field term.

Our analysis is based on the \emph{weak formulation} of stochastic optimal control problems, sometimes known as the \emph{martingale approach}; see for example, \cite{davismartingale, bsdefinance, pengcontrol}. This approach depends heavily on the non-degeneracy of $\sigma$ and its independence of the control, and in our case, it is also important that $\sigma$ has no mean field term. The \emph{strong} formulation of the problem, as in \cite{carmonadelarue-mfg}, would require that the state SDEs have strong solutions when controls are applied. The two formulations are compared in Remark \ref{strongformulation}. One of the main conveniences of the weak formulation is that weak existence and uniqueness of the state SDE require much less regularity in the coefficients, which are allowed to be path-dependent and merely measurable in the state variable. Also, the value function solves a backward stochastic differential equation (BSDE), and necessary and sufficient conditions for the optimality of a control follow easily from the comparison principle for BSDEs. This method is discussed by El Karoui and Quenez in \cite{bsdefinance}, Peng in \cite{pengcontrol}, and perhaps most thoroughly by Hamadene and Lepeltier in \cite{hamadenelepeltier-control}.

Our results allow for the mean field interaction (at least in the running reward function $f$) to occur through the control processes in addition to the state processes. This seems quite important for many practical applications and has received very little attention thusfar in the literature of mean field games. A very recent paper of Gomes and Voskanyan \cite{gomes-extendedmfg} uses PDE methods to study these types of interactions in the deterministic case, $\sigma \equiv 0$, under the name \emph{extended mean field games}. Under strong continuity and convexity assumptions, they obtain existence as well as some regularity of the solutions, and interestingly they are able to allow for general dependence of the running objective $f$ on the \emph{joint} law of the state and control processes. Our setting is very different: notably $\sigma > 0$, and our convexity and continuity assumptions are much weaker.

We also allow for very general nonlocal mean field interactions, including but not limited to weakly or Wasserstein continuous functionals. Among the natural interactions that have not yet been addressed in the mean field games literature which we are able to treat, we mention the case of coefficients which depend on the rank (Example \ref{example-rank} in Section \ref{section-discussion}), or on the mean field of the individual's nearest neighbors (Section \ref{section-flocking}). Our framework also includes models with different types of agents, similar to \cite{huangmfg1}. Moreover, $f$ does not need to be strictly convex, and may in fact be identically zero. A final novelty of our results worth emphasizing is that they apply in non-Markovian settings and require no continuity in the state variable.

For the sake of illustration, we present two applications which had been touted as models for mean field games, without being solved in full generality. First we study price impact models in which asset price dynamics depend naturally on the \emph{rates of change} of investors' positions, inspired by the model of Carlin et al. \cite{carlinloboviswanathan}.
As a second application of our theoretical results, we discuss a model of flocking proposed by Nourian et al. in \cite{nourian-cuckersmalemfg1} to provide a mechanism by which flocking behavior emerges as an equilibrium, as a game counterpart of the well-known Cucker-Smale model, \cite{cuckersmale}. In \cite{nourian-cuckersmalemfg1}, the authors identify the mean field limit and, \emph{under the assumption that there exists a unique solution to the limiting mean field game}, construct approximate Nash equilibria for the finite-player games. While flocking is often defined mathematically as a large time phenomenon (case in point, the stationary form of the mean field game strategy is considered in \cite{nourian-cuckersmalemfg1}), we treat the finite horizon case to be consistent with the set-up of the paper, even though this case is most often technically more challenging. We provide existence and approximation results for both their model and two related nearest-neighbor models.

This paper is organized as follows. We introduce the two practical applications in Section \ref{section-applications}. The price impact models of Section \ref{section-priceimpact} motivate the analysis of mean field games in which players interact through their controls, while Section \ref{section-flocking} describes the flocking model of \cite{nourian-cuckersmalemfg1} as well as two related nearest-neighbor models. Then, Section \ref{section-mfg} provides precise statements of the assumptions used throughout the paper and the main existence and uniqueness results. Section \ref{section-finite} explains the construction of approximate Nash equilibria for the finite-player game. The assumptions of the main theorems are discussed in more detail in Section \ref{section-discussion}, along with important examples. In Section \ref{section-applicationsrevisited} the general theory is specialized to the applications of Section \ref{section-applications}. The proofs of the main theorems of Sections \ref{section-mfg} and \ref{section-finite} are given in Sections \ref{section-mfgproof} and \ref{section-finiteproof}, respectively.

\section{Applications} \label{section-applications}
\subsection{Price impact models} \label{section-priceimpact}
To motivate our generalization of the class of mean field games worthy of investigation, we present a simple multi-agent model of price impact which leads to mean field interaction through the control processes. The model is along the lines of Almgren and Chriss's model \cite{chriss-almgren} for price impact, or rather its natural extension to an $n$-player competitive game given by Carlin, Lobo, and Viswanathan in \cite{carlinloboviswanathan}. The latter model is highly tractable, modeling a flat order book from which each agent must execute a fixed order. We instead model a nonlinear order book and use fairly general reward functions. See \cite{alfonsifruthschied,gatheralschiedslynko-transient} for a discussion of order book mechanics as well as a discussion of resilience, a concept we do not address. In our model, after each trade, the order book reconstructs itself instantly around a new mid-price $S_t$, and with the same shape. At each time $t$, each agent faces a cost structure given by the same transaction cost curve $c : \R \rightarrow [0,\infty]$, which is convex and satisfies $c(0) = 0$. We consider only order books with finite volume; an infinite value for $c(\alpha)$ simply means that the volume $\alpha$ is not available. Flat order books are common in the literature, though not realistic: they correspond to quadratic transaction costs $c$.

We work on a filtered probability space $(\Omega, \F, \mathbb{F}=(\F_t)_{t \in [0,T]},P)$ supporting $n+1$ independent Wiener processes, $W^1,\ldots,W^n$ and $B$. Let $S$ denote the asset price, $K^i$ the cash of agent $i$, and $X^i$ his position. Each agent controls his trading rate $\alpha^i_t$ and his position evolves according to
\begin{align*}
dX^i_t &= \alpha^i_tdt + \sigma dW^i_t.
\end{align*}
The noise term $\sigma dW^i_t$ models a random stream of demand that a broker may receive from his clients.
If a single agent $i$ places a market order of $\alpha^i_t$ when the mid-price is $S_t$, the transaction costs him $\alpha^i_t S_t + c(\alpha^i_t)$. Hence, the changes in cash of agent $i$ are naturally given by
\begin{align*}
dK^i_t &= -(\alpha^i_tS_t + c(\alpha^i_t))dt.
\end{align*}
Assuming $c$ is differentiable on its domain, the marginal price per share of this trade is $S_t + c'(\alpha^i_t)$, meaning that the agent receives all of the volume on the order book between the prices $S_t$ and $S_t + c'(\alpha^i_t)$. The order book should recenter somewhere in this price range, say at $S_t + \gamma c'(\alpha^i_t) / n$, where $\gamma > 0$. The factor of $1/n$ is irrelevant when $n$ is fixed, but it is the right scaling factor for obtaining a mean field approximation.

In a continuous-time, continuous-trading model with multiple agents, it is not clear how simultaneous trades should be handled. Somewhat more realistic are continuous-time, discrete-trade models, which many continuous-trade models are designed to approximate. In a continuous-time, discrete-trade model, it is reasonable to assume that agents never trade simultaneously, given that there is a continuum of trade times to choose from. We choose to model this in our continuous-trade setting in the following manner: When the $n$ agents trade at rates $\alpha^1_t,\ldots,\alpha^n_t$ at time $t$, agent $i$ still pays $\alpha^i_t S_t + c(\alpha^i_t)$, but the total change in price is
\[
\frac{\gamma}{n}\sum_{i=1}^nc'(\alpha^i_t).
\]
Finally, the mid-price is modeled as an underlying martingale plus a drift representing a form of permanent price impact:
\[
dS_t = \frac{\gamma}{n}\sum_{i=1}^nc'(\alpha^i_t)dt + \sigma_0dB_t.
\]
Note that the particular case $c(\alpha)=\alpha^2$ corresponds to the influential Almgren-Chriss model \cite{chriss-almgren}. The wealth $V^i_t$ of agent $i$ at time $t$, as pegged to the mid-price, is given by $V^i_0 + X^i_tS_t + K^i_t$, which leads to the following dynamics:
\begin{align}
dV^i_t = \left(\frac{\gamma}{n}\sum_{j=1}^nc'(\alpha^j_t)X^i_t - c(\alpha^i_t)\right)dt + \sigma_0X^i_tdB_t + \sigma S_t dW^i_t. \label{fo:dvit}
\end{align}
We assume that the agents are risk-neutral and seek to maximize their expected terminal wealths at the end of the trading period, including some agency costs given by functions $f$ and $g$, so that the objective of agent $i$ is to maximize:
\[
J^i = \E\left[ V^i_T - \int_0^Tf(t,X^i_t)dt - g(X^i_T) \right]. 
\]
Price impact models are most often used in optimal execution problems for high frequency trading. Because of their short time scale, the fact that $S_t$ as defined above can become negative is not an issue in practice. In these problems, one often chooses $g(x) = mx^2$ for some $m > 0$ in order to penalize left over inventory. The function $f$ is usually designed to provide an incentive for tracking a benchmark, say the frequently used market volume weighted average price (VWAP) and a penalty slippage.

If the control processes are square integrable and the cost function $c$ has at most quadratic growth, the volumes $X^i_t$ and the transaction price $S_t$ are also square integrable and the quadratic variation terms in \eqref{fo:dvit} are true martingales. So after using It\^{o}'s formula we find
\[
J^i = \E\left[\int_0^T\left(\frac{\gamma}{n}\sum_{j=1}^nc'(\alpha^j_t)X^i_t - c(\alpha^i_t) - f(t,X^i_t)\right)dt - g(X^i_T)\right].
\]
Treating $X^i$ as the state processes, this problem is of the form described in the introduction. The general theory presented in the sequel will apply to this model under modest assumptions on the functions $c$, $f$, and $g$, ensuring existence of approximate Nash equilibria. Intuitively, when $n$ is large, a single agent may ignore his price impact without losing much in the way of optimality. This model could be made more realistic in many ways, but we believe any improvement will preserve the basic structure of the price impact, which naturally depends on the mean field of the control processes. It should be mentioned that the risk-neutrality assumption is crucial and hides a much more difficult problem. Without risk-neutrality, we would have to keep track of $V$ and $S$ as state processes. More importantly, the Brownian motion $B$ would not disappear after taking expectations, and this would substantially complicate the mean field limit. 

\subsection{Flocking models} \label{section-flocking}
The position $X^i_t$ and velocity $V^i_t$ of individual $i$ change according to
\begin{align*}
dX^i_t &= V^i_tdt, \\
dV^i_t &= \alpha^i_tdt + \sigma dW^i_t,
\end{align*}
where $\alpha^i_t$ is the individual's acceleration vector, $W^i$ are independent $d$-dimensional Wiener processes, and $\sigma > 0$ is a $d \times d$ matrix (usually $d = 2$ or $d=3$). The objective of individual $i$  is to choose $\alpha^i$ to minimize
\begin{align}
\E\left[\int_0^T|\alpha^i_t|_R^2 + \left|\frac{1}{n}\sum_{j=1}^n(V^j_t - V^i_t)\phi(|X^j_t - X^i_t|)\right|_Q^2\,dt\right]. \label{flocking-objective}
\end{align}
Here, $\phi : [0,\infty) \rightarrow [0,\infty)$ is a nonincreasing function, and $|x|_Q := x^\top Qx$ and $|x|_R := x^\top Rx$ for $x \in \R^d$, where $Q$ and $R$ are positive semidefinite $d \times d$ matrices. The $|\alpha^i_t|_R^2$ term penalizes too rapid an acceleration, while the second term provides an incentive for an individual to align his velocity vector with the average velocity of the flock. The weights $\phi(|X^j_t - X^i_t|)$ emphasize the velocities of nearby (in position) individuals more than distant ones. In \cite{nourian-cuckersmalemfg1}, drawing inspiration from \cite{cuckersmale}, $\phi$ is of the form
\begin{align}
\phi(x) = c\left(1 + x^2\right)^{-\beta}, \ \ \beta \ge 0, \ c > 0. \label{cs-weight}
\end{align}

Our existence and approximation results apply to the model above as well as a related model in which the weights in \eqref{flocking-objective} take a different form. Namely, individual $i$ may give non-zero weight only to those individuals it considers to be neighbors, where the set of neighbors may be determined in two different ways. \emph{Nearest neighbor rules} pre-specify a radius $r > 0$, and an individual $i$'s neighbors at time $t$ are those individuals $j$ with $|X^j_t - X^i_t| \le r$. Letting $N^i_t$ denote the set of such $j$ and $|N^i_t|$ its cardinality, the objective function is
\begin{align}
\E\left[\int_0^T|\alpha^i_t|_R^2 + \left|\frac{c}{|N^i_t|}\sum_{j \in N^i_t}(V^j_t - V^i_t)\right|_Q^2\,dt\right]. \label{flocking-objective-nn}
\end{align}
This is inspired by what is now known as Vicsek's model, proposed in \cite{vicsek} and studied mathematically in \cite{jadbabaie-coordination}. On the other hand, recent studies such as \cite{ballerini-interaction} provide evidence that birds in flocks follow so-called \emph{k-nearest neighbor rules}, which track only a fixed number $k \le n$ of neighbors at each time. The corresponding objective function is the same, if we instead define $N^i_t$ to be the set of indices $j$ of the $k$ closest individuals to $i$ (so of course $|N^i_t| = k$). Note that there are no ``ties''; that is, for each distinct $i,j,l \le n$ and $t > 0$, we have $P(|X^i_t-X^j_t| = |X^i_t-X^l_t|) = 0$.

\section{Mean field games} \label{se-mfg}
\label{section-mfg}
We turn now to a general discussion of the mean field game models which we consider in this paper. We collect the necessary notation and assumptions in order to state the main existence, uniqueness, and approximation theorems.

\subsection{Construction of the mean field game} \label{section-construction}
Let $\B(E,\tau)$ denote the Borel $\sigma$-field of a topological space $(E,\tau)$. When the choice of topology is clear, we use the abbreviated form $\B(E)$. For a measurable space $(\Omega,\F)$, let $\P(\Omega)$ denote the set of probability measures on $(\Omega,\F)$. We write $\mu \ll \mu'$ when $\mu$ is absolutely continuous with respect to $\mu'$, and $\mu \sim \mu'$ when the measures are equivalent. Given a measurable function $\psi : \Omega \rightarrow [1,\infty)$, we set:
\begin{align*}
\P_\psi(\Omega) &= \left\{\mu \in \P(\Omega) : \int \psi \,d\mu < \infty \right\}, \\
B_\psi(\Omega) &= \left\{f : \Omega \rightarrow \R \text{ measurable, } \sup_\omega|f(\omega)|/\psi(\omega) < \infty \right\}.
\end{align*}
We define $\tau_\psi(\Omega)$ to be the weakest topology on $\P_\psi(\Omega)$ making the map $\mu \mapsto \int f\,d\mu$ continuous for each $f \in B_\psi(\Omega)$. The space $(\P_\psi(\Omega),\tau_\psi(\Omega))$ is generally neither metrizable nor separable, which will pose some problems. We define the empirical measure map  $e_n : \Omega^n \rightarrow \P(\Omega)$ by
\[
e_n(\omega_1,\ldots,\omega_n) = \frac{1}{n}\sum_{j=1}^n\delta_{\omega_j}.
\]
Notice that $e_n$ need not be $\B(\P_\psi(\Omega),\tau_\psi(\Omega))$-measurable, but this will not be an issue.

\begin{definition} \label{measurable}
Given measurable spaces $E$ and $F$, we say that a function $f : \P(\Omega) \times E \rightarrow F$ is \emph{empirically measurable} if
\[
\Omega^n \times E \ni (x,y) \mapsto f(e_n(x),y) \in F
\]
is jointly measurable for all $n \ge 1$.
\end{definition}

Let $\C := C([0,T];\R^d)$ be the space of $\R^d$-valued continuous functions on $[0,T]$ endowed with the sup-norm $\|x\| := \sup_{s \in [0,T]}|x(s)|$ and fix a Borel measurable function $\psi : \C \rightarrow [1,\infty)$ throughout. It will play a role similar to the ``Lyapunov-like'' function of G\"{a}rtner \cite{gartnermkv}, controlling a tradeoff between integrability and continuity requirements. Some comments on the choice of $\psi$ follow in Remark \ref{remark-psi}. For any $\mu \in \P(\C)$ and $t\in[0,T]$, the marginal $\mu_t$ denotes the image of $\mu$ under the coordinate map $\C \ni x \mapsto x_t \in \R^d$.

We use the notation $\lambda_0 \in \P(\R^d)$ for the initial distribution of the infinitely many players' state processes. Let $\Omega := \R^d \times \C$, define $\xi(x,\omega) := x$ and $W(x,\omega) := \omega$, and let $P$ denote the product of $\lambda_0$ and the Wiener measure, defined on $\B(\Omega)$. Define $\F_t$ to be the completion of $\sigma((\xi, W_s) : 0 \le s \le t)$ by $P$-null sets of $\B(\Omega)$, and set $\mathbb{F} := (\F_t)_{0 \le t \le T}$. We work with the filtered probability space $(\Omega,\F_T, \mathbb{F}, P)$ for the remainder of the section. For $k \in \N$ and $q \ge 1$ define the space $\mathbb{H}^{q,k}$ to be the set of progressively measurable $h : [0,T] \times \Omega \rightarrow \R^k$ satisfying
\[
\E\left[\left(\int_0^T|h_t|^2dt\right)^{q/2}\right] < \infty.
\]
For a martingale $M$, we denote by $\mathcal{E}(M)$ its Doleans stochastic exponential.
We now state assumptions on the data which will stand throughout the paper. Unless otherwise stated, $\P_\psi(\C)$ is equipped with the topology $\tau_\psi(\C)$.

The following assumptions \textbf{(S)} are implicitly assumed throughout the paper.
\begin{A1}[Standing assumptions] {\ }
\begin{enumerate}[(S.1)]
\item The \textit{control space} $A$ is a compact convex subset of a normed vector space, and the set $\A$ of \textit{admissible controls} consists of all progressively measurable $A$-valued processes. The volatility $\sigma : [0,T] \times \C \rightarrow \R^{d \times d}$ is progressively measurable. The drift $b : [0,T] \times \C \times \P_\psi(\C) \times A \rightarrow \R^d$ is such that $(t,x) \mapsto b(t,x,\mu,a)$ is progressively measurable for each $(\mu,a)$, and $a \mapsto b(t,x,\mu,a)$ is continuous for each $(t,x,\mu)$.
\item There exists a unique strong solution $X$ of the driftless state equation
\begin{align}
dX_t = \sigma(t,X)dW_t, \ \ X_0 = \xi, \label{stateSDE}
\end{align}
such that $\E[\psi^2(X)] < \infty$, $\sigma(t,X) > 0$ for all $t \in [0,T]$ almost surely, and $\sigma^{-1}(t,X)b(t,X,\mu,a)$ is uniformly bounded.
\end{enumerate}
\end{A1}

We will elaborate on these and the subsequent assumptions in Section \ref{section-applicationsrevisited} below, but for now let us make a few remarks. If $\sigma$ has linear growth, $\psi(x) = 1 + \|x\|^p$, and $\int_{\R^d}|x|^{2p}\lambda_0(dx) < \infty$, then indeed $\E[\psi^2(X)] < \infty$.
Compactness of $A$ is a strong assumption which will be used in several places, in particular to ensure that $\P(A)$ is compact.
Boundedness of $\sigma^{-1}b$ is also restrictive, but it will be crucial to ensure that the Hamiltonian is a uniformly Lipschitz function of the adjoint variable. See Remark \ref{re:bounded} for more details and some comments about relaxing these assumption.

From now on, $X$ denotes the unique solution of (\ref{stateSDE}). For each $\mu \in \P_\psi(\C)$ and $\alpha \in \A$, define a measure $P^{\mu,\alpha}$ on $(\Omega,\F_T)$ by
\[
\frac{dP^{\mu,\alpha}}{dP} = \mathcal{E}\left( \int_0^\cdot\sigma^{-1}b\left(t,X,\mu,\alpha_t\right) dW_t\right)_T .
\]
By Girsanov's theorem and boundedness of $\sigma^{-1}b$, the process $W^{\mu,\alpha}$ defined by
\[
W^{\mu,\alpha}_t := W_t - \int_0^t\sigma^{-1}b\left(s,X,\mu,\alpha_s\right)ds
\]
is a Wiener process under $P^{\mu,\alpha}$, and
\[
dX_t = b\left(t,X,\mu,\alpha_t\right)dt + \sigma(t,X) dW^{\mu,\alpha}_t.
\]
That is, under $P^{\mu,\alpha}$, $X$ is a weak solution of the state equation. Note that $P^{\mu,\alpha}$ and $P$ agree on $\F_0$; in particular, the law of $X_0 = \xi$ is still $\lambda_0$. Moreover, $\xi$ and $W$ remain independent under $P^{\mu,\alpha}$.

\begin{remark} \label{filtration}
It is well-known that the nonsingularity assumption (S.2) of $\sigma$ guarantees that $\mathbb{F}$ coincides with the completion of the filtration generated by $X$. It is thus implicit in the definition of $\A$ that our admissible controls can be written in \emph{closed-loop} form, that is as deterministic functions of $(t,X)$.
\end{remark}

We now state the assumptions on the \emph{reward} functions entering the objectives to be maximized by the players. Throughout, $\P(A)$ is endowed with the weak topology and its corresponding Borel $\sigma$-field.
\begin{enumerate}
\item[(S.3)] The running reward $f : [0,T] \times \C \times \P_\psi(\C) \times \P(A) \times A \rightarrow \R$ is such that $(t,x) \mapsto f(t,x,\mu,q,a)$ is progressively measurable for each $(\mu,q,a)$ and $a \mapsto f(t,x,\mu,q,a)$ is continuous for each $(t,x,\mu,q)$. The terminal reward function $g : \C \times \P_\psi(\C) \rightarrow \R$ is such that $x \mapsto g(x,\mu)$ is Borel measurable for each $\mu$.
\item[(S.4)] There exist $c > 0$ and an increasing function $\rho : [0,\infty) \rightarrow [0,\infty)$ such that
\[
|g(x,\mu)| + |f(t,x,\mu,q,a)| \le c\left(\psi(x) + \rho\left(\int \psi \, d\mu\right)\right), \quad \forall(t,x,\mu,q,a).
\]
Since $\psi \ge 1$, this is equivalent to the same assumption but with $\psi$ replaced by $1 + \psi$.
\item[(S.5)] The function $f$ is of the form
\[
f(t,x,\mu,q,a) = f_1(t,x,\mu,a) + f_2(t,x,\mu,q).
\]
\end{enumerate}

\begin{remark} \label{re:s5}
The only restrictive assumption among (S.3-5) is (S.5). Combined with the assumption that $b$ does not depend on $q$, assumption (S.5) renders the maximizer(s) of the Hamiltonian independent of the $\P(A)$ argument. Separation assumptions of this sort are common in mean field games literature, largely for this reason (c.f. \cite{lasrylionsmfg}).
\end{remark}

Given a measure $\mu \in \P_\psi(\C)$, a control $\alpha \in \A$, and a measurable map $[0,T] \ni t \mapsto q_t \in \P(A)$, we define the associated expected reward by
\[
J^{\mu,q}(\alpha) := \E^{\mu,\alpha} \left[\int_0^Tf(t,X,\mu,q_t,\alpha_t)dt + g(X,\mu) \right]
\]
where $\E^{\mu,\alpha}$ denotes expectation with respect to the measure $P^{\mu,\alpha}$. Considering $\mu$ and $q$ as fixed, we are faced with a standard stochastic optimal control problem, the value of which is given by
\[
V^{\mu,q} = \sup_{\alpha \in \A}J^{\mu,q}(\alpha).
\]

\begin{definition}
We say a measure $\mu \in \P_\psi(\C)$ and a measurable function $q : [0,T] \rightarrow \P(A)$ form a \emph{solution of the MFG} if there exists $\alpha \in \A$ such that $V^{\mu,q} = J^{\mu,q}(\alpha)$, $P^{\mu,\alpha} \circ X^{-1} = \mu$, and $P^{\mu,\alpha} \circ \alpha_t^{-1} = q_t$ for almost every $t$.
\end{definition}

\subsection{Existence and uniqueness} 
\label{section-existenceuniqueness}
Some additional assumptions are needed for the existence and uniqueness results. Define the Hamiltonian $h : [0,T] \times \C \times \P_\psi(\C) \times \P(A) \times \R^d \times A \rightarrow \R$, the maximized Hamiltonian $H : [0,T] \times \C \times \P_\psi(\C) \times \P(A) \times \R^d \rightarrow \R$, and the set on which the supremum is attained by
\begin{align}
h(t,x,\mu,q,z,a) &:= f(t,x,\mu,q,a) + z \cdot \sigma^{-1}b(t,x,\mu,a), \nonumber \\
H(t,x,\mu,q,z) &:= \sup_{a \in A}h(t,x,\mu,q,z,a), \label{hdef} \\
A(t,x,\mu,q,z) &:= \{a \in A : h(t,x,\mu,q,z,a) = H(t,x,\mu,q,z)\}, \nonumber
\end{align}
respectively. Note that $A(t,x,\mu,q,z)$ does not depend on $q$, in light of assumption (S.5), so we shall often drop $q$ from the list of arguments of $A$ and use the notation $A(t,x,\mu,z)$. Note also that $A(t,x,\mu,z)$ is always nonempty, since $A$ is compact and $h$ is continuous in $a$ by assumptions (S.1) and (S.3).

\begin{A41}
For each $(t,x,\mu,z)$, the set $A(t,x,\mu,z)$ is convex.
\end{A41}

It will be useful to have notation for the driftless law and the set of equivalent laws:
\begin{align*}
\X &:= P \circ X^{-1} \in \P_\psi(\C) \\
\P_X &:= \left\{\mu \in \P_\psi(\C) : \mu \sim \X\right\}.
\end{align*}
\begin{A2}[Existence assumptions] 
For each $(t,x) \in [0,T]\times \C$ the following maps are sequentially continuous, using $\tau_\psi(\C)$ on $\P_X$ and the weak topology on $\P(A)$:
\begin{alignat*}{5}
\P_X \times A &\ni (\mu,a) \; &\mapsto \; &b(t,x,\mu,a), \\
\P_X \times \P(A) \times A &\ni (\mu,q,a) \; &\mapsto \; &f(t,x,\mu,q,a), \\
\P_X &\ni \mu \; &\mapsto \; &g(x,\mu).
\end{alignat*}
\end{A2}

\begin{theorem} \label{existence}
Suppose (E) and (C) hold. Then there exists a solution of the MFG.
\end{theorem}

\begin{remark} \label{sequential}
It is worth emphasizing that sequential continuity is often easier to check for $\tau_\psi(\C)$, owing in part to the failure of the dominated convergence theorem for nets. For example, functions like
\[
\mu \mapsto \int\int\phi(x,y)\mu(dx)\mu(dy)
\]
for bounded measurable $\phi$ are always sequentially continuous but may fail to be continuous.
\end{remark}

\begin{remark} \label{remark-psi}
The function $\psi$ enters the assumptions in two essential ways. On the one hand, the functions $b$, $f$, and $g$ should be $\tau_\psi(\C)$-continuous in their measure arguments as in (E). On the other hand, the solution of the SDE $dX_t = \sigma(t,X)dW_t$ should possess $\psi^2$-moments as in (S.2), and the growth of $f$ and $g$ should be controlled by $\psi$, as in (S.4). There is a tradeoff in the choice of $\psi$: larger $\psi$ makes the latter point more constraining and the former less constraining.
\end{remark}

The following uniqueness theorem is inspired by Lasry and Lions \cite{lasrylionsmfg}. They provide counterexamples to show that one should not expect uniqueness in much generality, unless one assumes that the time horizon is small and the coefficients are suitably Lipschitz (e.g. \cite{huangmfg1}).

\begin{A42} {\ }
\begin{enumerate}[(U.1)]
\item For each $(t,x,\mu,z)$, the set $A(t,x,\mu,z)$ is a singleton;
\item $b = b(t,x,a)$ has no mean field term;
\item $f(t,x,\mu,a) = f_1(t,x,\mu) + f_2(t,\mu,q) + f_3(t,x,a)$ for some $f_1$, $f_2$, and $f_3$;
\item For all $\mu, \mu' \in \P_\psi(\C)$, 
\begin{align*}
\!\!\!\!\!\!\!\!\!\!\!\!\!\int_\C\left[g(x,\mu)-g(x,\mu') + \int_0^T\left(f_1(t,x,\mu)-f_1(t,x,\mu')\right)dt\right](\mu-\mu')(dx) \le 0.
\end{align*}
\end{enumerate}
\end{A42}

\begin{theorem} \label{uniqueness}
Suppose (U) holds. Then there is at most one solution of the MFG.
\end{theorem}

\begin{corollary} \label{existence-uniqueness}
Suppose (E) and (U) hold. Then there exists a unique solution of the MFG.
\end{corollary}

\begin{remark} \label{re:extrarandomness}
The following simple extension of the above formulation allows more heterogeneity among agents. Work instead on a probability space $\Omega = \Omega' \times \R^d \times \C$, where $\Omega'$ is some measurable space which will model additional time-zero randomness. We may then fix an initial law $\lambda_0 \in \P(\Omega' \times \R^d)$, and let $P$ be the product of $\lambda_0$ and Wiener measure. Letting $(\theta,\xi,W)$ denote the coordinate maps, we work with the filtration generated by the process $(\theta,\xi,W_s)_{0 \le s \le T}$. The data $b$, $\sigma$, $f$, and $g$ may all depend on $\theta$. In the finite-player game, the agents have i.i.d. initial data $(\theta^i,\xi^i)$, known at time zero, where $\xi^i$ is the initial state and $\theta^i$ can encode other differences between the agents. For example, in a price impact model, perhaps a fraction $\rho \in [0,1]$ of the agents need to liquidate but the rest do not; this can be modeled using such a $\theta$ which equals $c > 0$ with probability $\rho$ and $0$ otherwise, and setting $g(X,\theta) = \theta|X_T|^2$ for some $c > 0$. This generalization complicates the notation but changes essentially none of the analysis.
\end{remark}

\section{Approximate Nash equilibria for finite-player games} \label{section-finite}
Before proving these theorems, we discuss how a solution of the MFG may be used to construct an approximate Nash equilibrium for the finite-player game, using only distributed controls. Additional assumptions are needed for the approximation results:

\begin{A3} {\ }
\begin{enumerate}[(F.1)]
\item $b = b(t,x,a)$ has no mean field term;
\item For all $(t,x,\mu,q,a)$, $f(t,x,\mu,q,a) = f(t,x,\mu^t,q,a)$, where $\mu^t$ denotes the image of $\mu$ under the map $\C \ni x \mapsto x_{\cdot \wedge t} \in \C$;
\item The functions $b$, $f$, and $g$ are empirically measurable, as in Definition \ref{measurable}, using the progressive $\sigma$-field on $[0,T] \times \C$, and Borel $\sigma$-fields elsewhere;
\item For each $(t,x)$, the following functions are continuous at each point satisfying $\mu \sim \X$:
\begin{alignat*}{5}
\P_\psi(\C) \times \P(A) \times A &\ni (\mu,q,a) \; &\mapsto \; &f(t,x,\mu,q,a), \\
\P_\psi(\C) &\ni \mu \; &\mapsto \; &g(x,\mu);
\end{alignat*}
\item There exists $c > 0$ such that, for all $(t,x,\mu,q,a)$,
\[
|g(x,\mu)| + |f(t,x,\mu,q,a)| \le c\left(\psi(x) + \int \psi \, d\mu\right).
\]
\end{enumerate}
\end{A3}

\begin{remark}
The continuity assumption (F.4) is stronger than assumption (E). Indeed, in (E) we required only \emph{sequential} continuity on a subset of the space $\P_\psi(\C)$. Assumption (F.2) is simply progressive measurability of $f$ with respect to the measure argument, which in fact was not needed for the results of Section \ref{se-mfg}. Analogs of the result of this section are possible when (F.1) fails, under stronger continuity requirements. Namely, $\sigma^{-1}b$, $f$, and $g$ should be continuous in $\mu$ uniformly in the other arguments, and $\sigma^{-1}b$ should be uniformly Lipschitz in $\mu$ with respect to total variation. However, we refrain from elaborating on this result, as it seems suboptimal and the proof is quite long.
\end{remark}

Adhering to the philosophy of the weak formulation, we choose a single convenient probability space on which we define the $n$-player games, simultaneously for all $n$. Assumptions (C) and (F) stand throughout this section (as does (S), as always). We fix a solution of the MFG $(\hat{\mu},\hat{q})$ throughout, whose existence is guaranteed by Theorem \ref{existence}, with corresponding closed-loop control $\hat{\alpha}(t,x)$ (see Remark \ref{filtration}). Consider a probability space $(\Omega,\F,P)$ supporting a sequence $(W^1,W^2,\ldots)$ of independent $d$-dimensional Wiener processes, independent $\R^d$-valued random variables $(\xi^1,\xi^2,\ldots)$ with common law $\lambda_0$, and processes $(X^1,X^2,\ldots)$ satisfying
\[
dX^i_t = b(t,X^i,\hat{\alpha}(t,X^i))dt + \sigma(t,X^i)dW^i_t, \ X^i_0 = \xi^i.
\]
For each $n$, let $\mathbb{F}^n = (\F^n_t)_{t \in [0,T]}$ denote the completion of the filtration generated by $(X^1,\ldots,X^n)$ by null sets of $\F$. Let $\mathbb{X}^i$ denote the completion of the filtration generated by $X^i$. Note that $X^i$ are independent and identically distributed and that the process $(\xi^i,W^i_t)_{0 \le t \le T}$ generates the same filtration $\mathbb{X}^i$, as in Remark \ref{filtration}. Abbreviate $\alpha^i_t = \hat{\alpha}(t,X^i)$. These controls are known as distributed controls.

We now describe the $n$-player game for fixed $n$. The control space $\A_n$ is the set of all $\mathbb{F}^n$-progressively measurable $A$-valued processes; the players have complete information of the other players' state processes. On the other hand, $\A^n_n$ is the $n$-fold Cartesian product of $\A_n$, or the set of $\mathbb{F}^n$-progressively measurable $A^n$-valued processes. Let $\mu^n$
denote the empirical measure of the first $n$ state processes as defined in the introduction by \eqref{fo:empirical}. For $\beta = (\beta^1,\ldots,\beta^n) \in \A^n_n$, define a measure $P_n(\beta)$ on $(\Omega,\F^n_T)$ by the density
\[
\frac{dP_n(\beta)}{dP} := \mathcal{E}\left(\int_0^\cdot\sum_{i=1}^n\left(\sigma^{-1}b(t,X^i,\beta^i_t) - \sigma^{-1}b(t,X^i,\alpha^i_t)\right) dW^i_t\right)_T.
\]
Under $P_n(\beta)$, for each $i=1,\ldots,n$, $X^i$ is a weak solution of the SDE
\[
dX^i_t = b(t,X^i,\beta^i_t)dt + \sigma(t,X^i)dW^{\beta^i,i}_t,
\]
where 
\[
W^{\beta^i,i}_\cdot := W_\cdot - \int_0^\cdot\left[\sigma^{-1}b(t,X^i,\beta^i_t) - \sigma^{-1}b(t,X^i,\alpha^i_t)\right]dt
\]
is a $d$-dimensional $P_n(\beta)$-Wiener process. Note that $X^i_0$ are i.i.d. with common law $\lambda_0$ under any of the measures $P_n(\beta)$ with $\beta \in \A_n^n$. For $\beta = (\beta^1,\ldots,\beta^n) \in \A^n_n$, the value to player $i$ of the strategies $\beta$ is defined by
\[
J_{n,i}(\beta) := \E^{P_n(\beta)}\left[\int_0^Tf(t,X^i,\mu^n,q^n(\beta_t),\beta^i_t)dt + g(X^i,\mu^n) \right],
\]
where, for $a = (a^1,\ldots,a^n) \in A^n$, we define
\[
q^n(a) := \frac{1}{n}\sum_{i=1}^n\delta_{a^i}.
\]
Note that the joint measurability assumption (F.3) guarantees that $g(X^i,\mu^n)$ is $\F^n_T$-measurable, while (F.2) and (F.3) ensure that $f(t,X^i,\mu^n,q^n(\beta_t),\beta^i_t)$ and $b(t,X^i,\beta^i_t)$ are progressively measurable with respect to $\mathbb{F}^n$.

\begin{theorem} \label{approximationtheorem}
Assume (C) and (F) hold, and let $(\hat{\mu},\hat{q})$ denote a solution of the MFG, with corresponding closed-loop control $\hat{\alpha} = \hat{\alpha}(t,x)$ (see Remark \ref{filtration}). Then the strategies $\alpha^i_t := \hat{\alpha}(t,X^i)$ form an approximate Nash equilbrium for the finite-player game in the sense that there exists a sequence $\epsilon_n \ge 0$ with $\epsilon_n \rightarrow 0$ such that, for $1 \le i \le n$ and $\beta \in \A_n$, 
\[
J_{n,i}(\alpha^1,\ldots,\alpha^{i-1},\beta,\alpha^{i+1},\ldots,\alpha^n) \le J_{n,i}(\alpha^1,\ldots,\alpha^n) + \epsilon_n.
\]
\end{theorem}

\begin{remark}
The punchline is that $\alpha^i$ is $\mathbb{X}^i$-adapted for each $i$. That is, player $i$ determines his strategy based only on his own state process. As explained earlier, such strategies are said to be \emph{distributed}. The theorem tells us that even with full information, there is an approximate Nash equilibrium consisting of distributed controls, and we know precisely how to construct one using a solution of the MFG. Note that the strategies $(\alpha^i)_{i \in \N}$ also form an approximate Nash equilibrium for any partial-information version of the game, as long as player $i$ has access to (at least) the filtration $\mathbb{X}^i$ generated by his own state process.
\end{remark}

\section{Discussion of the assumptions and examples} \label{section-discussion}
This section discusses some important special cases of the assumptions of Sections \ref{se-mfg} and \ref{section-finite}. Assumptions (C) and (U) are examined first, before we turn to assumptions (S), (E), and (F).

\subsection{Assumptions (C) and (U)}
Condition (C) (resp. (U.1)) is crucial for the fixed point (resp. uniqueness) argument and holds when the Hamiltonian $h(t,x,\mu,q,z,a)$ is concave (resp. strictly concave) in $a$, for each $(t,x,\mu,q,z)$, which is a common assumption in control theory. For example, condition (C) (resp. (U.1)) holds if $b$ is affine in $a$ and $f$ is concave (resp. strictly concave) in $a$. More generally, we can get away with quasiconcavity in the previous statements. Note that if $f \equiv 0$ then $A(t,x,\mu,0) = A$, and thus condition (U.1) fails except in trivial cases. However, condition (C) frequently holds even in the absence of a running reward function $f \equiv 0$; the optimal control in such a case is typically a bang-bang control.

\begin{example}[Monotone functionals of measures]
Here we provide some examples of the monotonicity assumption (U.4) of Theorem \ref{uniqueness}. For any of the following $g$, we have
\[
\int_\C\left[g(x,\mu)-g(x,\mu')\right](\mu-\mu')(dx) \le 0, \ \forall \mu,\mu' \in \P_\psi(\C).
\]
\begin{itemize}
\item $g(x,\mu) = \phi_1(x) + \phi_2(\mu)$ for some $\phi_1 : \C \rightarrow \R$ and $\phi_2 : \P_\psi(\C) \rightarrow \R$. In this case, there is equality for all $\mu,\mu'$.
\item $g(x,\mu) = \left|\phi(x) - \int_\C \phi(y)\mu(dy) \right|^2$ for some $\phi : \C \rightarrow \R$. If, for example, $\phi(x) = x$, then this payoff function rewards a player if his state process deviates from the average.
\item $g(x,\mu) = -\int_{\R^d}\phi(|x-y|)\mu_T(dy)$, where $\phi : [0,\infty) \rightarrow [0,\infty)$ is bounded, continuous, and positive definite. A special case is when $\phi$ is bounded, nonincreasing, and convex; see Proposition 2.6 of \cite{gatheralschiedslynko-transient}.
\end{itemize}
\end{example}

\subsection{Assumptions (S), (E), and (F)}
Standard arguments give:
\begin{lemma} 
\label{sigmagrowthlemma1}
Assume that $\psi_0 : \R^d \rightarrow [1,\infty)$ is either $\psi_0(x) = 1 + |x|^p$ for some $p \ge 1$ or $\psi_0(x) = e^{p|x|}$ for some $p>0$, and let $\psi(x) = \sup_{t \in [0,T]}\psi_0(x_t)$. If $\int_{\R^d}\psi_0(x)^2\lambda_0(dx) < \infty$, $\sigma > 0$, $|\sigma(\cdot,0)| \in L^2[0,T]$, and $|\sigma(t,x) - \sigma(t,y)| \le c\|x-y\|$ for some $c > 0$, then (S.2) holds as long as $\sigma^{-1}b$ is bounded.
\end{lemma}
The measurability requirement (F.3) is unusual, but not terribly restrictive. The more difficult assumption to verify is that of continuity, (F.4). Common assumptions in the literature involve continuity with respect to the topology of weak convergence or more generally a Wasserstein metric. For a separable Banach space $(E,\|\cdot\|_E)$ and $p \ge 1$, let
\[
\W^p_{E,p}(\mu,\mu') := \inf_\pi\int_E\|x-y\|_E^p\pi(dx,dy),
\]
where the infimum is over all $\pi \in \P(E \times E)$ with marginals $\mu$ and $\mu'$. When $\psi_{E,p}(x) = 1 + \|x\|_E^p$, it is known that $\W_{E,p}$ metrizes the weakest topology making the map $\P_{\psi_{E,p}}(E) \ni \mu \mapsto \int\phi\,d\mu$ continuous for each \emph{continuous} function $\phi \in B_{\psi_{E,p}}(E)$ (see Theorem 7.12 of \cite{villanibook}). Thus $\W_{E,p}$ is weaker than $\tau_{\psi_{E,p}}(\C)$, which proves the following result.

\begin{lemma} \label{sigmagrowthlemma2}
Let $\psi = \psi_{\C,p}$, $p \ge 1$. Suppose $f$ and $g$ are (sequentially) continuous in $(\mu,q,a)$ at points with $\mu \sim \X$, for each $(t,x)$, using the metric $\W_{\C,p}$ on $\P_\psi(\C)$. Then (F.4) holds.
\end{lemma}

In most applications the coefficients are Markovian; that is, 
\[
f(t,x,\mu,q,a) = \hat{f}(t,x_t,\mu_t,q,a), \text{ for some } \hat{f}.
\]
Note that for any $\mu,\mu' \in \P(\C)$, $p \ge 1$, and $t \in [0,T]$,
\[
\W_{\R^d,p}(\mu_t,\mu'_t) \le \W_{\C,p}(\mu,\mu'),
\]
and thus the previous proposition includes Markovian data. Note also that assumption (F.4) demands continuity in the measure argument only at the points which are equivalent to $\X$. Of course, if $\sigma$ does not depend on $X$ or is uniformly bounded from below, then $\X_t \sim \L$ for all $t > 0$, and thus in the Markovian case we need only to check that $\hat{f}$ is continuous at points which are equivalent to Lebesgue measure. At no point was a Markov property of any use, and this is why we chose to allow path-dependence in each of the coefficients. Moreover, continuity in the spatial variable was never necessary either. Indeed, we require only that $dX_t = \sigma(t,X)dW_t$ admits a strong solution, as in assumption (S.2), which of course covers the usual Lipschitz assumption. The most common type of mean field interaction is scalar and Markovian, so we investigate such cases carefully.

\begin{proposition}[Scalar dependence on the measure] \label{proplinear}
Consider a function of the form
\[
f(t,x,\mu,q,a) = \int_\C F(t,x_t,y_t,q,a) \mu(dy) = \int_{\R^d} F(t,x_t,y,q,a) \mu_t(dy)
\]
where $F : [0,T] \times \R^d \times \R^d \times \P(A) \times A \rightarrow \R$ is jointly measurable and jointly continuous in its last two arguments whenever the first three are fixed. Let $\psi_0 : \R^d \rightarrow [1,\infty)$ be lower semicontinuous, and suppose there exists $c > 0$ such that
\[
\sup_{(t,a) \in [0,T] \times A}|F(t,x,y,q,a)| \le c(\psi_0(x) + \psi_0(y))
\]
for all $(x,y) \in \R^d \times \R^d$. Let $\psi(x) = \sup_{t \in [0,T]}\psi_0(x_t)$ for $x \in \C$. Then $f$ satisfies the relevant parts of assumptions (S.3), (S.4), (E), (F).
\end{proposition}
\begin{proof}
Note that $\psi : \C \rightarrow [1,\infty)$ is lower-semicontinuous and thus measurable. Note also that the function $\C \ni y \mapsto F(t,x,y_t,q,a) \in \R$ is in $B_\psi(\C)$ for each $(t,x,q,a) \in [0,T] \times \R^d \times \P(A) \times A$, and thus $f$ is indeed well defined for $\mu \in \P_\psi(\C)$. Property (F.2) is obvious, and property (F.5) follows from the inequality
\[
|f(t,x,\mu,q,a)| \le c\left(\psi_0(x_t) + \int_\C\psi_0(y_t)\mu(dy)\right).
\]
The measurability assumption (F.3) is easy to verify. Condition (E) will follow from (F.4), which we prove now. 

Fix $(t,x) \in [0,T] \times \C$, and let $E = \P(A) \times A$. Let $F_0(y,\eta) := F(t,x_t,y,\eta)$ for $(y,\eta) \in \R^d \times E$. Fix $(\mu,\eta) \in \P_\psi(\C) \times E$ and a net $(\mu^\alpha,\eta^\alpha)$ converging to $(\mu,\eta)$. We also have $\mu^\alpha_t \rightarrow \mu_t$ in $\tau_{\psi_0}(\R^d)$. Note that
\begin{align*}
f(t,x,\mu^\alpha,\eta^\alpha) - f(t,x,\mu,\eta) = &\int_{\R^d} (F_0(y,\eta^\alpha) - F_0(y,\eta))\mu^\alpha_t(dy) \\
	&+ \int_{\R^d} F_0(y,\eta)(\mu^\alpha_t - \mu_t)(dy)
\end{align*}
The second term clearly tends to zero. For the first term, fix $\epsilon > 0$. Since $E$ is compact metric, the function $\R^d \ni y \mapsto F_0(y,\cdot) \in C(E)$ is measurable, using the Borel $\sigma$-field generated by the supremum norm on the space $C(E)$ of continuous real-valued functions of $E$; see Theorem 4.55 of \cite{aliprantisborder}. Thus, by Lusin's theorem (12.8 of \cite{aliprantisborder}), there exists a compact set $K \subset \R^d$ such that $\int_{K^c}\psi_0\,d\mu_t < \epsilon$ and $K \ni y \mapsto F_0(y,\cdot) \in C(E)$ is continuous. Since $|F_0(y,\eta^\prime)| \le c(\psi_0(x_t) + \psi_0(y))$ for all $(y,\eta^\prime) \in \R^d \times E$, 
\begin{align*}
\left|\int_{\R^d} (F_0(y,\eta^\alpha) - F_0(y,\eta))\mu^\alpha_t(dy)\right| \le &\sup_{y \in K}|F_0(y,\eta^\alpha) - F_0(y,\eta)| \\
	&+ 2c\int_{K^c}(\psi_0(x_t) + \psi_0(y))\,\mu^\alpha_t(dy).
\end{align*}
It follows from the compactness of $E$ and Lemma \ref{lemma-joint-continuity} below that the restriction of $F_0$ to $K \times E$ is uniformly continuous. Since $K$ is compact, we use Lemma \ref{lemma-joint-continuity} again in the other direction to get $\sup_{y \in K}|F_0(y,\eta^\alpha) - F_0(y,\eta)| \rightarrow 0$. Since also
\[
\lim\int_{K^c}(\psi_0(x_t) + \psi_0(y))\,\mu^\alpha_t(dy) = \int_{K^c}(\psi_0(x_t) + \psi_0(y))\,\mu_t(dy) \le (1 + \psi_0(x_t))\epsilon,
\]
we have
\[
\limsup\left|\int_{\R^d} (F_0(y,\eta^\alpha) - F_0(y,\eta))\mu^\alpha_t(dy)\right| \le 2c(1 + \psi_0(x_t))\epsilon.
\]
\end{proof}

\begin{corollary} \label{scalar-corollary}
Let $F$ and $\psi_0$ be as in Proposition \ref{proplinear}, and suppose
\[
f(t,x,\mu,q,a) = G\left(t,x_t,\int_{\R^d} F(t,x_t,y,q,a) \mu_t(dy),q,a\right),
\]
where $G : [0,T] \times \R^d \times \R \times \P(A) \times A \rightarrow \R$ is jointly measurable and continuous in its last three arguments. If also
\[
|G(t,x,y,q,a)| \le c\left(\psi_0(x) + |y|\right)
\]
for some $c > 0$, then $f$ satisfies the relevant parts of assumptions (S.3), (S.4), (E), (F).
\end{corollary} 

We will occasionally need the following simple lemma, which was used in the proof of Proposition \ref{proplinear}. Its proof is straightforward and thus omitted. 

\begin{lemma} \label{lemma-joint-continuity}
Let $E$ and $K$ be topological spaces with $K$ compact, let $G : E \times K \rightarrow \R$, and let $x_0 \in E$ be fixed. Then $G$ is jointly continuous at points of $\{x_0\} \times K$ if and only if $G(x_0,\cdot)$ is continuous and $x \mapsto \sup_{y \in K}|G(x,y) - G(x_0,y)|$ is continuous at $x_0$. 
\end{lemma}

\begin{example}[Geometric Brownian motion] \label{ex:geometricBM}
Requiring $\sigma^{-1}b$ to be \\ bounded rather than $\sigma^{-1}$ and $b$ each to be bounded notably allows for state processes of a geometric Brownian motion type. For example, if $d=1$, our assumptions allow for coefficients of the form
\begin{align*}
b(t,x,\mu,a) &= \hat{b}(t,\mu,a)x_t, \\
\sigma(t,x) &= \hat{\sigma}(t)x_t,
\end{align*}
where $\hat{\sigma}(t) > 0$ for all $t$ and $\hat{\sigma}^{-1}\hat{b}$ is bounded.
\end{example}

\begin{remark}
\label{re:bounded}
We close the subsection with a remark on the assumption of boundedness of $\sigma^{-1}b$, which could certainly be relaxed. The reason for this assumption lies in the BSDE (\ref{BSDEvalue}) for the value function; boundedness of $\sigma^{-1}b$ equates to a standard Lipschitz driver, as covered in \cite{pardouxpengbsde}. The results of Hamadene and Lepeltier in \cite{hamadenelepeltier-control} may be applied if $b$ and $\sigma$ have linear growth in $x$ and $\sigma$ is bounded below, but this increases the technicalities and rules out a direct application of the results of \cite{hupeng-stability}. However, we only really need \cite{hupeng-stability} in order to treat mean field interactions in the control, and thus our analysis should still work under appropriate linear growth assumptions. Our assumptions of boundedness of $\sigma^{-1}b$ and compactness of $A$ unfortunately rule out common linear-quadratic models, but, nonetheless, the same general techniques could be used to study a large class of linear-quadratic problems (still, of course, with uncontrolled volatility) in which both these assumptions fail. More care is required in the choice of admissible controls, and the BSDE for the value function becomes quadratic in $z$; this program was carried out for stochastic optimal control problems in \cite{fuhrmanhutessitore-control}, and could presumably be adapted to mean field games.
\end{remark}

\subsection{Additional Examples}
Corollary \ref{scalar-corollary} allows us to treat many mean field interactions which are not weakly continuous, as they may involve integrals of discontinuous functions. This is useful in the following examples.

\begin{example}[Rank effects] \label{example-rank}
Suppose an agent's reward depends on the rank of his state process among the population. That is, suppose $d=1$ and $f(t,x,\mu,q,a)$ involves a term of the form $G(\mu_t(-\infty,x_t])$, where $G : [0,1] \rightarrow \R$ is continuous. Such terms with $G$ monotone are particularly interesting for applications, as suggested for a model of oil production in \cite{gueantlasrylionsmfg}. The intuition is that an oil producer prefers to produce before his competitors, in light of the uncertainty about the longevity of the oil supply. The state process $X$ represents oil reserves, and $G$ should be decreasing in their model. Proposition \ref{proplinear} shows that the inclusion of such terms as $\mu_t(-\infty,x_t]$ in $f$ or $g$ is compatible with all of our assumptions. If $b$ contains such rank effects, no problem is posed for assumptions (S) and (E), but of course (F.1) is violated.
\end{example}

\begin{example}[Types]
In \cite{huangmfg1}, Huang, Caines, and Malham\'{e} consider multiple types of agents, and a dependence on the mean field within each type. The number of types is fixed, and an agent cannot change type during the course of the game. Using the construction of Remark \ref{re:extrarandomness}, we may model this by giving each agent a random but i.i.d. type at time zero. Alternatively, in some models an agent's type may change with his state (or with time, or with his strategy); for example, a person's income bracket depends on his wealth. Suppose, for example, that $A_1,A_2,\ldots,A_m \subset \R^d$ are Borel sets of positive Lebesgue measure, and define $F_i : \P(\R^d) \rightarrow \P(\R^d)$ by $F_i(\mu)(B) := \mu(B \cap A_i) / \mu(A_i)$ when $\mu(A_i) > 0$ and $F_i(\mu) = 0$ otherwise. As long as $\sigma$ is bounded away from zero, then $\X_t \sim \L$ where $\L$ is again Lebesgue measure on $\R^d$, and indeed $F_i$ are $\tau_1(\R^d)$-continuous at points $\mu \sim \X_t$. So we can treat functionals of the form
\[
f(t,x,\mu,q,a) = G(t,x_t,F(\mu_t),q,a),
\]
where $F = (F_1,\ldots,F_m)$, and $G : [0,T] \times \R^d \times (\P(\R^d))^m \times \P(A) \times A \rightarrow \R$.
\end{example}

\section{Applications revisited} \label{section-applicationsrevisited}
Before proving the main results, we return briefly to the models presented in Section \ref{section-applications}, for which we demonstrate the applicability of the existence and approximation theorems (\ref{existence} and \ref{approximationtheorem}).

\subsection{Price impact models}
We restrict our attention to finite-volume order books. We suppose that $A \subset \R$ is a compact interval containing the origin, $c' : A \rightarrow \R$ is continuous and nondecreasing, $\sigma > 0$, $f : [0,T] \times \R \rightarrow \R$ and $g : \R \rightarrow \R$ are measurable, and finally that there exists $c_1 > 0$ such that
\[
|f(t,x)| + |g(x)| \le c_1 e^{c_1|x|}, \text{  for all } (t,x)\in [0,T]\times\R.
\]
Let $c(x) = \int_0^xc'(a)da$. Assume that $X^i_0$ are i.i.d. and that their common distribution $\lambda_0 \in \P(\R)$ satisfies $\int_\R e^{p|x|}\lambda_0(dx) < \infty$ for all $p > 0$. In the notation of the paper, we have $b(t,x,\mu,a) = a$, $\sigma(t,x) = \sigma$, $f(t,x,\mu,q,a) = \gamma x_t\int_Ac'dq - c(a) - f(t,x_t)$, $g(x,\mu) = g(x_T)$, and $\psi(x) = e^{c_1\|x\|}$.

It is quite easy to check the assumptions of the previous sections, at least with the help of Lemma \ref{sigmagrowthlemma1} below, yielding the following theorem. Moreover, in this simple case we can estimate the rate of convergence, as proven at the end of Section \ref{section-finiteproof}.

\begin{proposition} \label{pr:priceimpact}
Under the above assumptions, the existence and approximation theorems \ref{existence} and \ref{approximationtheorem} apply to the price impact model. Moreover, in the approximation theorem, there exists a constant $C > 0$ such that
\[
\epsilon_n \le C/\sqrt{n}.
\]
\end{proposition}

\subsection{Flocking models}
To work around the degeneracy of the diffusion $(X^i,V^i)$, we consider only $V^i$ as the state variable, and recover $X^i$ by making the coefficients path-dependent. Let $b(t,v,\mu,a) = a$, $\sigma > 0$ constant, $g \equiv 0$, and $A \subset \R^d$ compact convex. Define $\iota : [0,T] \times \C \rightarrow \R^d$ and $I : [0,T] \times \P(\C) \rightarrow \P(\R^d)$ by
\[
\iota(t,v) := \int_0^tv_sds, \quad I(t,\mu) := \mu \circ \iota(t,\cdot)^{-1}.
\]
Note that $\iota(t,V^i)$ represents the position of the individual at time $t$; we are assuming each individual starts at the origin to keep the notation simple and consistent, although any initial distribution of positions could be accounted for by using the construction of Remark \ref{re:extrarandomness}. For flocking models, \eqref{flocking-objective} is captured by choosing a running reward function of the form:
\[
f^{(1)}(t,v,\mu,a) = -|\alpha|_R^2 - \left|\int_\C\mu(dv')(v'_t-v_t)\phi(|\iota(t,v'-v)|)\right|^2_Q.
\]
The minus signs are only to turn the problem into a maximization, to be consistent with the notation of the rest of the paper. Recall that $\phi : [0,\infty) \rightarrow [0,\infty)$ is nonincreasing and thus Borel measurable. Assume the initial data $V^i$ are i.i.d. and square-integrable, with law $\lambda_0 \in \P_2(\R^d)$. Take $\psi(x) = 1 + \|x\|^2$ for $x \in \C$. For the nearest neighbor model, we use
\begin{align*}
&f^{(2)}(t,v,\mu,a) \\
	& \ = -|\alpha|_R^2 - \left|\frac{c}{I(t,\mu)(B(\iota(t,v),r))}\int_\C\mu(dv')(v'_t-v_t)1_{B(\iota(t,v),r)}(\iota(t,v'))\right|^2_Q,
\end{align*}
where $r > 0$ was given, and $B(x,r')$ denotes the closed ball of radius $r'$ centered at $x$. Consider the second term above to be zero whenever $I(t,\mu)(B(\iota(t,v),r)) = 0$. Finally, for the $k$-nearest-neighbor model, we choose $\eta \in (0,1)$ to represent a fixed percentage of neighbors, which amounts to keeping $k/n$ fixed in the finite-player game as we send $n \rightarrow \infty$. We define $r : \P(\R^d) \times \R^d \rightarrow [0,\infty)$ by 
\[
r(\mu,x,y) := \inf\left\{r^\prime > 0 : \mu(B(x,r^\prime)) \ge y\right\},
\]
and
\begin{align*}
&f^{(3)}(t,v,\mu,a) \\
	& \ = -|\alpha|_R^2 - \left|\frac{c}{\eta}\int_\C\mu(dv')(v'_t-v_t)1_{B(\iota(t,v),r(I(t,\mu),\iota(t,v)),\eta)}(\iota(t,v'))\right|^2_Q.
\end{align*}
It is straightforward to check that the objective \eqref{flocking-objective-nn} for the nearest neighbor models is equivalent to maximizing
\[
\E\int_0^Tf^{(1)}(t,V^i,\mu^n,\alpha^i_t)dt, \quad \text{ where } \mu^n = \frac{1}{n}\sum_{j=1}^n\delta_{V^j},
\]
replacing $f^{(1)}$ by $f^{(2)}$ in the case of the $k$-nearest neighbor model.

\begin{proposition} \label{pr:flocking}
Under the above assumptions, the existence and approximation theorems \ref{existence} and \ref{approximationtheorem} apply to each of the flocking models.
\end{proposition}
\begin{proof}
Assumptions (S.1), (S.4), (S.5), (C), (F.1), (F.2), and (F.5) are easy to check. Lemma \ref{sigmagrowthlemma1} below takes care of (S.2). Also, (S.3) and (F.3) are clear for $f^{(1)}$ and $f^{(2)}$, and follow from Lemma \ref{lemma-radius} below for $f^{(3)}$. It remains to check the continuity assumption (F.4). For $f^{(1)}$, this follows from Proposition \ref{proplinear} below. Apply It\^{o}'s formula to $tW_t$ to get
\[
\iota(t,X) = \int_0^tX_sds = \int_0^t(\xi + \sigma W_s)ds = t\xi + \sigma t W_t - \sigma\int_0^tsdW_s.
\]
Since $\xi$ and $W$ are independent, we see that $I(t,\X) \sim \L$ for $t \in (0,T]$, where $\L$ denotes Lebesgue measure on $\R^d$. Hence $I(t,\mu) \sim \L$ for $\mu \sim \X$, and so $\mu \mapsto 1/I(t,\mu)(B(x,r))$ is $\tau_\psi(\C)$-continuous at points $\mu \sim \X$, for each $(x,r) \in \R^d \times (0,\infty)$. This along with Proposition \ref{proplinear} below establish (F.4) for $f^{(2)}$. Finally, we prove (F.4) for $f^{(3)}$. Fix $(t,v) \in (0,T] \times \C$, and define
\begin{align*}
B_\mu &:= B\left(\iota(t,v),r(I(t,\mu),\iota(t,v),\eta)\right), \\
F(\mu) &:= \int_\C(v'_t - v_t)1_{B_\mu}(\iota(t,v'))\mu(dv'),
\end{align*}
for $\mu \in \P_\psi(\C)$. In light of Lemma \ref{sigmagrowthlemma2} and the discussion preceding it, it suffices to show $F$ is $\W_{\C,2}$-continuous at points $\mu \sim \X$. Let $\mu^n \rightarrow \mu$ in $\W_{\C,2}$ with $\mu \sim \X$, and note that $I(t,\mu) \sim I(t,\X) \sim \L$. Then
\begin{align*}
F(\mu^n) - F(\mu) = &\int_\C(v'_t - v_t)\left(1_{B_{\mu^n}} - 1_{B_\mu}\right)(\iota(t,v'))\mu(dv') \\
	&+ \int_\C(v'_t - v_t)1_{B_{\mu^n}}(\iota(t,v'))[\mu^n-\mu](dv') \\
	=: I_n + II_n.
\end{align*}
Note that $I(t,\mu^n) \rightarrow I(t,\mu)$ weakly, and thus
\[
r(I(t,\mu^n),\iota(t,x),\eta) \rightarrow r(I(t,\mu),\iota(t,x),\eta)
\]
by Lemma \ref{lemma-radius}. Since $1_{B_{\mu^n}} \rightarrow 1_{B_\mu}$ holds $\L$-a.e. (and thus $I(t,\mu)$-a.e.) and $\int_\C(v'_t-v_t)[\mu^n-\mu](dv') \rightarrow 0$, the dominated convergence theorem yields $I_n \rightarrow 0$. To show $II_n \rightarrow 0$, note that
note that 
\[
I(t,(v'_t - v_t)\mu^n(dv')) \rightarrow I(t,(v'_t - v_t)\mu(dv')), \text{ weakly}.
\]
Since the latter measure is absolutely continuous with respect to Lebesgue measure, Theorem 4.2 of \cite{rao-weakuniformconvergence} implies
\[
II_n = \left[I(t,(v'_t - v_t)\mu^n(dv')) - I(t,(v'_t - v_t)\mu(dv'))\right](B_{\mu^n}) \rightarrow 0.
\]
In fact, we should consider separately the positive and negative parts of each of the $d$ components of the signed vector measures $(v'_t - v_t)\mu^n(dv')$, since Theorem 4.2 of \cite{rao-weakuniformconvergence} is stated only for nonnegative real-valued measures.
\end{proof}

\begin{lemma} \label{lemma-radius}
The function $r$ is empirically measurable, and $r(\cdot,x,y)$ is weakly continuous at points $\mu \sim \L$.
\end{lemma}
\begin{proof}
To prove measurability, note that for any $c > 0$
\begin{align*}
\left\{(z,x,y) : r(e_n(z),x,y) > c\right\} = \left\{(z,x,y) : \frac{1}{n}\sum_{i=1}^n1_{B(x,c)}(z_i) < y\right\}
\end{align*}
is clearly a Borel set in $(\R^d)^n \times \R^d \times (0,1)$ for each $n$. To prove continuity, let $\mu_n \rightarrow \mu$ weakly in $\P(\R^d)$ with $\mu \sim \L$. Let $\epsilon > 0$. Since $\mu \sim \L$, the map $r \mapsto \mu(B(x,r))$ is continuous and strictly increasing. Thus the inverse function $r(\mu,x,\cdot)$ is also continuous, and we may find $\delta > 0$ such that $|r(\mu,x,y) - r(\mu,x,z)| < \epsilon$ whenever $|z-y| \le \delta$. Theorem 4.2 of \cite{rao-weakuniformconvergence} tells us that $\mu_n(B) \rightarrow \mu(B)$ uniformly over measurable convex sets $B$, since $\mu \ll \L$. Hence, for $n$ sufficiently large,
\[
\sup_{(x,r) \in \R^d \times (0,\infty)}\left|\mu(B(x,r)) - \mu_n(B(x,r)) \right| < \delta.
\]
Thus, for sufficiently large $n$,
\begin{align*}
r(\mu_n,x,y) &= \inf\left\{r' > 0 : \mu(B(x,r')) \ge y + (\mu-\mu_n)(B(x,r'))\right\} \\
	&\ge \inf\left\{r' > 0 : \mu(B(x,r')) \ge y - \delta\right\} \\
	&= r(\mu,x,y-\delta) \ge r(\mu,x,y) - \epsilon,
\end{align*}
and similarly
\[
r(\mu_n,x,y) \le \inf\left\{r' > 0 : \mu(B(x,r')) \ge y + \delta \right\} = r(\mu,x,y + \delta) \le r(\mu,x,y) + \epsilon.
\]
\end{proof}

\section{Proofs of existence and uniqueness theorems} \label{section-mfgproof}
This section is devoted to the proofs of the existence and uniqueness results of Theorems \ref{existence} and \ref{uniqueness}. Throughout the section, we work with the canonical probability space described in the second paragraph of Section \ref{section-mfg}. Since BSDEs will be used repeatedly, it is important to note that the classical existence, uniqueness, and comparison results for BSDEs do indeed hold in our setting, despite the fact that $\mathbb{F}$ is not the Brownian filtration. The purpose of working with the Brownian filtration is of course for martingale representation, which we still have with our slightly larger filtration: It follows from Theorem 4.33 of \cite{jacodshiryaevbook}, for example, that every square integrable $\mathbb{F}$-martingale $(M_t)_{0 \le t \le T}$ admits the representation $M_t = M_0 + \int_0^t \phi_s dW_s$ for some $\phi \in \mathbb{H}^{2,d}$. However, note that in our case the initial value of the solution of a BSDE is random, since $\F_0$ is not trivial.

To find a fixed point for the law of the control, we will make use of the space $\M$ of positive Borel measures $\nu$ on $[0,T] \times \P(A)$ (using the weak topology on $\P(A)$) whose first projection is Lebesgue measure; that is, $\nu([s,t] \times \P(A)) = t-s$ for $0 \le s \le t \le T$. Endow $\M$ with the weakest topology making the map $\nu \mapsto \int\phi\,d\nu$ continuous for each bounded measurable function $\phi : [0,T] \times \P(A) \rightarrow \R$ for which $\phi(t,\cdot)$ is continuous for each $t$. This is known as the stable topology, which was studied thoroughly by Jacod and M\'{e}min in \cite{jacodmemin-stable}. In particular, since $A$ is a compact metrizable space, so is $\P(A)$, and thus so is $\M$. Note that a measure $\nu \in \M$ disintegrates into $\nu(dt,dq) = \nu_t(dq)dt$, where the measurable map $[0,T] \ni t \mapsto \nu_t \in \P(\P(A))$ is uniquely determined up to almost everywhere equality. For any bounded measurable function $F : \P(A) \rightarrow \R^k$, we extend $F$ to $\P(\P(A))$ in the natural way by defining
\[
F(\nu) := \int_{\P(A)}\nu(dq)F(q).
\]
In this way, $F(\delta_q) = F(q)$ for $q \in \P(A)$.

\begin{remark}
Because of condition (S.5), the aforementioned convention will not lead to any confusion regarding the meaning of $H(t,x,\mu,\nu,z)$, for $\nu \in \P(\P(A))$. In particular, it is consistent with the relationship
\[
H(t,x,\mu,\nu,a) := \sup_{a \in A}h(t,x,\mu,\nu,z,a),
\]
since the only dependence of $h$ on $\nu$ is outside of the supremum.
\end{remark}

For each $(\mu,\nu) \in \P_\psi(\C) \times \M$, we now construct the corresponding control problem. The standing assumptions (S) are in force throughout, and the following construction is valid without any of the other assumptions. Recall the definitions of $h$ and $H$ from (\ref{hdef}) in Section \ref{section-mfg}. That $(t,x,z) \mapsto H(t,x,\mu,\nu_t,z)$ is jointly measurable for each $(\mu,\nu)$ follows, for example, from the measurable maximum Theorem 18.19 of \cite{aliprantisborder}. Boundedness of $\sigma^{-1}b$ guarantees that $H$ is uniformly Lipschitz in $z$. Since $\mu \in \P_\psi(\C)$, it follows from assumptions (S.2) and (S.4) that $g(X,\mu) \in L^2(P)$ and that $(H(t,X,\mu,\nu_t,0))_{0 \le t \le T} = (\sup_af(t,X,\mu,\nu_t,a))_{0 \le t \le T} \in \mathbb{H}^{2,1}$. Hence the classical result of Pardoux and Peng \cite{pardouxpengbsde} (or rather a slight extension thereof, as remarked above) applies, and there exists a unique solution $(Y^{\mu,\nu},Z^{\mu,\nu}) \in \mathbb{H}^{2,1} \times \mathbb{H}^{2,d}$ of the BSDE
\begin{align}
Y^{\mu,\nu}_t = g(X,\mu) + \int_t^TH(s,X,\mu,\nu_s, Z^{\mu,\nu}_s)ds - \int_t^TZ^{\mu,\nu}_sdW_s. \label{BSDEvalue}
\end{align}
For each $\alpha \in \A$, we may similarly solve the BSDE
\begin{align*}
Y^{\mu,\nu,\alpha}_t &= g(X,\mu) + \int_t^Th(s,X,\mu,\nu_s, Z^{\mu,\nu,\alpha}_s,\alpha_s)ds - \int_t^TZ^{\mu,\nu,\alpha}_sdW_s \\
	&= g(X,\mu) + \int_t^Tf(s,X,\mu,\nu_s,\alpha_s)ds - \int_t^TZ^{\mu,\nu,\alpha}_sdW^{\mu,\alpha}_s.
\end{align*}
Since $W^{\mu,\alpha}$ is a Wiener process under $P^{\mu,\alpha}$ and $Y^{\mu,\alpha}$ is adapted, we get
\[
Y^{\mu,\nu,\alpha}_t = \E^{\mu,\alpha}\left[\left.g(X,\mu) + \int_t^Tf(s,X,\mu,\nu_s,\alpha_s)ds\right| \F_t\right].
\]
In particular, $\E[Y^{\mu,\nu,\alpha}_0] = J^{\mu,\nu}(\alpha)$.

It is immediate from the comparison principle for BSDEs (e.g. Theorem 2.2 of \cite{bsdefinance}) that $\E[Y^{\mu,\nu}_0] \ge \E[Y^{\mu,\nu,\alpha}_0] = J^{\mu,\nu}(\alpha)$ for each $\alpha \in \A$, and thus $\E[Y^{\mu,\nu}_0] \ge V^{\mu,\nu}$. By a well-known measurable selection theorem (e.g. Theorem 18.19 of \cite{aliprantisborder}), there exists a function $\hat{\alpha} : [0,T] \times \C \times \P_\psi(\C) \times \R^d \rightarrow A$ such that
\begin{align}
\hat{\alpha}(t,x,\mu,z) \in A(t,x,\mu,z), \quad \text{ for all } (t,x,\mu,z), \label{selection}
\end{align}
and such that for each $\mu$ the map $(t,x,z) \mapsto \hat{\alpha}(t,x,\mu,z)$ is jointly measurable with respect to the progressive $\sigma$-field on $[0,T] \times \C$ and $\B(\R^d)$. Letting 
\begin{align}
\alpha^{\mu,\nu}_t := \hat{\alpha}(t,X,\mu,Z^{\mu,\nu}_t),  \label{selection2}
\end{align}
the uniqueness of solutions of BSDEs implies $Y^{\mu,\nu}_t = Y^{\mu,\nu,\alpha^{\mu,\nu}}_t$, which in turn implies $V^{\mu,\nu} = J^{\mu,\nu}(\alpha^{\mu,\nu})$ since $J^{\mu,\nu}(\alpha^{\mu,\nu}) \le V^{\mu,\nu}$.

The process $\alpha^{\mu,\nu}$ is an optimal control, but so is any process in the set
\begin{align}
\A(\mu,\nu) := \left\{ \alpha \in \A : \alpha_t \in A(t,X,\mu,Z^{\mu,\nu}_t) \ dt \times dP-a.e.\right\}. \label{alphamudef}
\end{align}
Define $\Phi : \P_\psi(\C) \times \A \rightarrow \P(\C) \times \M$ by
\[
\Phi(\mu,\alpha) := (P^{\mu,\alpha} \circ X^{-1}, \delta_{P^{\mu,\alpha} \circ \alpha^{-1}_t}(dq)dt)
\]
The goal now is to find a point $(\mu,\nu) \in \P_\psi(\C) \times \M$ for which there exists $\alpha \in \A(\mu,\nu)$ such that $(\mu,\nu) = \Phi(\mu,\alpha)$. In other words, we seek a fixed point of the set-valued map $(\mu,\nu) \mapsto \Phi(\mu,\A(\mu,\nu)) := \{\Phi(\mu,\alpha) : \alpha \in \A(\mu,\nu)\}$. Note that under condition (U.1), $\alpha^{\mu,\nu}$ is the unique element of $\A(\mu,\nu)$ (up to almost everywhere equality), and this reduces to a fixed point problem for a single-valued function.

\begin{remark} \label{nonempty}
It is worth emphasizing that the preceding argument demonstrates that the set $\A(\mu,\nu)$ is always nonempty, under only the standing assumptions (S).
\end{remark}

\begin{remark} \label{zdifficulty}
The main difficulty in the analysis is the adjoint process $Z^{\mu,\nu}$. Note that for each $(\mu,\nu)$ there exists a progressively measurable function $\zeta_{\mu,\nu} : [0,T] \times \C \rightarrow \R^d$ such that $Z^{\mu,\nu}_t = \zeta_{\mu,\nu}(t,X)$. If we choose a measurable selection $\hat{\alpha}$ as in (\ref{selection}), any weak solution of the following McKean-Vlasov SDE provides a solution of the MFG:
\[
\begin{cases}
dX_t \!\!\!\!\!&= b(t,X,\mu,\hat{\alpha}(t,X,\mu,\zeta_{\mu,\nu}(t,X)))dt + \sigma(t,X)dW_t, \\
X &\sim \mu, \ \mu \circ (\hat{\alpha}(t,\cdot,\mu,\zeta_{\mu,\nu}(t,\cdot)))^{-1} = \nu_t \ a.e.
\end{cases}
\]
The notation $X \sim \mu$ means that $\mu$ should equal the law of $X$. This map $\zeta_{\mu,\nu}$ is typically quite inaccessible, which is why we do not appeal to any existing results on McKean-Vlasov equations, even when $\nu$ is not present. All such results require some kind of continuity of the map
\[
(x,\mu) \mapsto b(t,x,\mu,\hat{\alpha}(t,x,\mu,\zeta_{\mu,\nu}(t,x))),
\]
as far as the authors know. It is possible to make assumptions on the data which would guarantee, for example, that $\zeta_{\mu,\nu}(t,\cdot)$ is continuous, but continuous dependence on $\mu$ would be a much trickier matter.
\end{remark}

\subsection{Some results of set-valued analysis}
We precede the main proofs with some useful lemmas. Without assumption (U), the optimal controls need not be unique, and thus we will need a fixed point theorem for set-valued maps. We first summarize some terminology from set-valued analysis.

For a point $y$ in a metric space $(E,d)$ and $\delta > 0$, let $B(y,\delta)$ denote the open ball of radius $\delta$ centered at $y$. Similarly, for $F \subset E$, let $B(F,\delta) = \{x \in E : \inf_{y \in F}d(x,y) < \delta\}$. For two subsets $F,G$ of $E$, we (abusively) define
\[
d(F,G) := \sup_{y \in G}d(F,y) = \sup_{y \in G}\inf_{x \in F}d(x,y).
\]
Note that $d$ is not symmetric. If $K$ is another metric space, a set-valued function $\Gamma : K \rightarrow 2^E$ is said to be \emph{upper hemicontinuous} at $x \in K$ if for all $\epsilon > 0$ there exists $\delta > 0$ such that $\Gamma(B(x,\delta)) \subset B(\Gamma(x),\epsilon)$. It is straightforward to prove that $\Gamma$ is upper hemicontinuous at $x \in K$ if and only if $d(\Gamma(x),\Gamma(x_n)) \rightarrow 0$ for every sequence $x_n$ converging to $x$.

In order to relax somewhat the convexity assumption of Kakutani's fixed point theorem, we adapt results of Cellina in \cite{cellina-approximation} to derive a slight generalization of Kakutani's theorem, which will assist in the proof of Theorem \ref{existence}.

\begin{proposition} \label{fixedpointtheorem}
Let $K$ be a compact convex metrizable subset of a locally convex topological vector space, and let $E$ be a normed vector space. Suppose $\Gamma : K \rightarrow 2^E$ is upper hemicontinuous and has closed and convex values, and suppose $\phi : K \times E \rightarrow K$ is continuous. Then there exists $x \in K$ such that $x \in \phi(x,\Gamma(x)) := \{\phi(x,y) : y \in \Gamma(x)\}$.
\end{proposition}
\begin{proof}
Let $Gr(\Gamma) := \{(x,y) \in K \times E : y \in \Gamma(x)\}$ denote the graph of $\Gamma$. By Cellina's result (Theorem 1 of \cite{cellina-approximation}), for each positive integer $n$ we may find a continuous (singe-valued) function $\gamma_n : K \rightarrow E$ such that the graph of $\gamma_n$ is contained in the $1/n$ neighborhood of $Gr(\Gamma)$. That is, for all $x \in K$,
\[
d((x,\gamma_n(x)),Gr(\Gamma)) := \inf\left\{d\left((x,\gamma_n(x)), (y,z)\right) : y \in K, z \in \Gamma(y)\right\} < 1/n,
\]
where $d$ denotes some metric on $K \times E$. Since $K \ni x \mapsto \phi(x,\gamma_n(x)) \in K$ is continuous, Schauder's fixed point theorem implies that there exists $x_n \in K$ such that $x_n = \phi(x_n,\gamma_n(x_n))$. By Lemma 17.8 and Theorem 17.10 of \cite{aliprantisborder}, $\Gamma(K) := \bigcup_{x \in K}\Gamma(x) \subset X$ is compact and $Gr(\Gamma)$ is closed. Thus $Gr(\Gamma) \subset K \times \Gamma(K)$ is compact. Since $d((x_n,\gamma_n(x)),Gr(\Gamma)) \rightarrow 0$ and $Gr(\Gamma)$ is compact, there exist a subsequence $x_{n_k}$ and a point $(x,y) \in Gr(\Gamma)$ such that $(x_{n_k},\gamma_{n_k}(x_{n_k})) \rightarrow (x,y)$. This completes the proof, since $y \in \Gamma(x)$ and since continuity of $\phi$ yields
\[
x = \lim x_{n_k} = \lim \phi(x_{n_k},\gamma_{n_k}(x_{n_k})) = \phi(x,y).
\]
\end{proof}

A special case of Berge's maximum theorem (17.31 of \cite{aliprantisborder}) will be useful:

\begin{theorem}[Berge's Theorem] \label{maximizer}
Let $E$ be a metric space, $K$ a compact metric space, and $\phi : E \times K \rightarrow \R$ a continuous function. Then $\gamma(x) := \max_{y \in K}\phi(x,y)$ is continuous, and the following set-valued function is upper hemicontinuous and compact-valued:
\[
E \ni x \mapsto \arg\max_{y \in K}\phi(x,y) := \{y \in K : \gamma(x) = \phi(x,y)\} \in 2^K
\]
\end{theorem}

\subsection{Proof of Theorem \ref{existence} (existence)} \label{section-existenceproof}
We now turn toward the proof of Theorem \ref{existence}. In what follows, we always use the topology $\tau_\psi(\C)$ on $\P_\psi(\C)$, except when stated otherwise. Despite its simplicity, we state the following result as a lemma for later references.  

\begin{lemma} \label{RND}
Let $(E,\mathcal{E})$ and $(F,\F)$ be measurable spaces, and let $\mu,\nu \in \P(E)$ with $\nu \ll \mu$. If $X:E \rightarrow F$ is measurable, then 
\[
\frac{d\nu \circ X^{-1}}{d\mu \circ X^{-1}} \circ X = \E^\mu\left[\left. \frac{d\nu}{d\mu} \right| X \right] \ \ \mu-a.s.
\]
\end{lemma}

\begin{lemma} \label{RNDbound}
For any $q \in \R$ with $|q| \ge 1$, we have (recall that $\X := P \circ X^{-1}$)
\begin{align}
M_q &:= \sup_{(\mu,\alpha) \in \P_\psi(\C) \times \A}\int\left(d\Phi(\mu,\alpha)/d\X\right)^qd\X < \infty. \label{mdef}
\end{align}
\end{lemma}
\begin{proof}
Recall that $\sigma^{-1}b$ is bounded, say by $c > 0$. Fix $(\mu,\alpha) \in \P_\psi(\C) \times \A$. Letting $N_t := \int_0^t\sigma^{-1}b(t,X,\mu,\alpha_t)dW_t$, we see that $[N,N]_T \le Tc^2$, and thus, since $q(q-1) \ge 0$,
\[
\mathcal{E}(N)_T^q = \mathcal{E}(qN)_T\exp\left(q(q-1)[N,N]_T/2\right) \le \mathcal{E}(qN)_T\exp\left(q(q-1)Tc^2/2\right).
\]
Hence, Lemma \ref{RND} and Jensen's inequality yield
\begin{align*}
\int\left(d\Phi(\mu,\alpha)/d\X\right)^qd\X &= \E\left[\E\left[\left.dP^{\mu,\alpha}/dP\right| X\right]^q\right] \le \E\left[(dP^{\mu,\alpha}/dP)^q\right] \\
	&\le \exp\left(q(q-1)Tc^2/2\right).
\end{align*}
Since this bound is independent of $(\mu,\alpha)$, we indeed have $M_q < \infty$.
\end{proof}

In terms of the notation from Lemma \ref{RNDbound}, let $M := \max(M_2,M_{-1})$. Let
\begin{align}
\Q := \left\{\mu \in \P_\psi(\C) : \mu \sim \X, \ \int(d\mu/d\X)^2d\X \le M, \ \int(d\X/d\mu)d\X \le M\right\}. \label{qdef}
\end{align}
By construction, the range of $\Phi$ is contained in $\Q \times \M$. Critical to our fixed point theorem is the following compactness result, which probably exists in various forms elsewhere in the literature. Part of the result may be found, for example, in Lemma 6.2.16 of \cite{dembozeitouni}. But, for lack of a concise reference, and to keep the paper fairly self-contained, we include a complete proof of the following:

\begin{proposition} \label{Qcompact}
The space $(\Q,\tau_\psi(\C))$ is convex, compact, and metrizable. Moreover, $\tau_1(\C)$ and $\tau_\psi(\C)$ induce the same topology on $\Q$.
\end{proposition}
\begin{proof}
Of course, by $\tau_1(\C)$ we mean $\tau_{\phi}(\C)$ with $\phi \equiv 1$. Define
\begin{align*}
\Q_1 &= \left\{\mu \in \P(\C) : \mu \ll \X, \ \int(d\mu/d\X)^2d\X \le M\right\}, \\
\Q_2 &= \left\{\mu \in \P(\C) : \mu \sim \X, \ \int(d\X/d\mu)d\X \le M\right\}.
\end{align*}
Cleary each set is convex. We will show that $\Q_1$ is compact and metrizable under $\tau_1(\C)$, that $\Q_2$ is $\tau_1(\C)$-closed, and that $\tau_1(\C)$ and $\tau_\psi(\C)$ induce the same topology on $\Q_1$.

Let $q \in \R$ with $|q| \ge 1$. The set $K_q := \{Z \in L^1(\X) : Z \ge 0 \ \X-a.s., \ \int |Z|^qd\X \le M\}$ is clearly convex. It is also norm-closed: if $Z_n \rightarrow Z$ in $L^1(\X)$ with $Z_n \in K_q$, then $Z_n \rightarrow Z$ $\X$-a.s. along a subsequence, and thus Fatou's lemma yields $\int |Z|^qd\X \le \liminf\int |Z_n|^qd\X \le M$. Hence, $K_q$ is weakly closed (see Theorem 5.98 of \cite{aliprantisborder}). For $q > 1$, the set $K_q$ is uniformly integrable and thus weakly compact, by the Dunford-Pettis theorem; moreover, $K_q$ is metrizable, since it is a weakly compact subset of separable Banach space (Theorem V.6.3 of \cite{dunfordschwartz1}). Now, for $\mu \ll \X$, define $F(\mu) := d\mu/d\X$. Then $F$ is a homeomorphism from $(\Q_2,\tau_1(\C))$ to $K_{-1}$ equipped with the weak topology of $L^1(\X)$, and so $\Q_2$ is $\tau_1(\C)$-closed. Similarly, $F$ is a homeomorphism from $(\Q_1,\tau_1(\C))$ to $K_2$ with the weak topology, and so $(\Q_1,\tau_1(\C))$ is compact and metrizable.

It remains to prove that $\tau_1(\C)$ and $\tau_\psi(\C)$ coincide on $\Q_1$. Let $\phi \in B_\psi(\C)$ with $|\phi| \le \psi$, $\mu \in \P_\psi(\C)$, and $\epsilon > 0$, and define $U = \{\nu \in \P_\psi(\C) : |\int\phi\,d(\nu - \mu)| < \epsilon\}$. Since $\tau_\psi(\C)$ is stronger than $\tau_1(\C)$, it suffices to find a $\tau_1(\C)$-neighborhood $V$ of $\mu$ with $V \cap \Q_1 \subset U \cap \Q_1$. First, note that for any $c>0$ and $\nu \in \Q_1$, the Cauchy-Schwarz inequality yields
\[
\left(\int_{\{\psi \ge c\}}\psi\,d\nu\right)^2 \le \int\left(\frac{d\nu}{d\X}\right)^2d\X\int_{\{\psi \ge c\}}\psi^2d\X \le M\int_{\{\psi \ge c\}}\psi^2d\X.
\]
Since $\int\psi^2d\X < \infty$ by (S.2), we may find $c > 0$ such that $\int_{\{\psi \ge c\}}\psi\,d\nu \le \epsilon/3$ for all $\nu \in \Q_1$. Then, for any $\nu \in \Q_1$,
\begin{align*}
\left|\int\phi\,d(\nu - \mu)\right| &\le \left|\int_{\{\psi < c\}}\phi\,d(\nu - \mu)\right| + \left|\int_{\{\psi \ge c\}}\phi\,d\nu\right| + \left|\int_{\{\psi \ge c\}}\phi\,d\mu\right| \\
	&\le \frac{2\epsilon}{3} + \left|\int_{\{\psi < c\}}\phi\,d(\nu - \mu)\right|.
\end{align*}
Set $V = \{\nu \in \P(\C) : |\int_{\{\psi < c\}}\phi\,d(\nu - \mu)| < \epsilon/3\}$, so that $V \cap \Q_1 \subset U \cap \Q_1$. Since $|\phi| \le \psi$, we have $\phi 1_{\{\psi < c\}} \in B_1(\C)$, and thus $V \in \tau_1(\C)$.
\end{proof}

The next two lemmas pertain to the $Z^{\mu,\nu}$ terms that arise in the BSDE representations above; in particular, a kind of continuity of the map $(\mu,\nu) \mapsto Z^{\mu,\nu}$ is needed. 

\begin{lemma} \label{continuity}
Suppose assumption (E) holds. Then for each $(t,x) \in [0,T] \times \C$, the function $\Q \times \P(A) \times \R^d \ni (\mu,q,z) \mapsto H(t,x,\mu,q,z)$ is continuous, and the set-valued function $\Q \times \R^d \ni (\mu,z) \mapsto A(t,x,\mu,z)$ is upper hemicontinuous.
\end{lemma}
\begin{proof}
Since $\Q$ is metrizable by Lemma \ref{Qcompact}, this is simply a combination of assumption (E) with Theorem \ref{maximizer}, using $E = \Q \times \P(A) \times \R^d$ and $K = A$. Recall from (S.1) that $A$ is compact.
\end{proof}

\begin{lemma} \label{zconvergence}
Suppose assumption (E) holds. Suppose $(\mu^n,\nu^n) \rightarrow (\mu,\nu)$ in $\Q \times \M$, using $\tau_\psi(\C)$ on $\Q$. Then
\[
\lim_{n \rightarrow \infty}\E\left[\int_0^T\left|Z^{\mu^n,\nu^n}_t - Z^{\mu,\nu}_t\right|^2dt\right] = 0.
\]
\end{lemma}
\begin{proof}
Note that the functions $H(s,x,\mu^\prime,\nu^\prime,\cdot)$ have the same Lipschitz constant for each $(t,x,\mu^\prime,\nu^\prime)$, coinciding with the uniform bound for $\sigma^{-1}b$. Assumption (S.4) implies
\begin{align*}
\E\left[\int_0^T|H(t,X,\mu^n,\nu^n_t,0)|^2dt\right] &= \E\left[\int_0^T\sup_{a \in A}|f(t,X,\mu^\prime,\nu^\prime_t,a)|^2dt\right] \\
	&\le 2c^2T\E[\psi^2(X)] + 2c^2T\rho^2\left(\int\psi\,d\mu^n\right)
\end{align*}
for all $1 \le n \le \infty$, where $(\mu^\infty,\nu^\infty) := (\mu,\nu)$. Since $\mu^n \in \P_\psi(\C)$ and $\mu^n \rightarrow \mu$ in $\tau_\psi(\C)$ it follows that $\sup_n\int\psi\,d\mu^n < \infty$. Since $\rho$ is increasing and nonnegative,
\begin{align}
\sup_n\rho^2\left(\int \psi \, d\mu^n\right)  = \rho^2\left(\sup_n\int\psi\,d\mu^n\right) < \infty. \label{zconv1}
\end{align}
Assumption (S.2) yields $\E[\psi^2(X)] < \infty$. Hence, we will be able to conclude via a convergence result for BSDEs proven by Hu and Peng in \cite{hupeng-stability}, as soon as we show that
\[
I_n := \E\left[\left|g(X,\mu^n) - g(X,\mu)\right|^2\right] \rightarrow 0,
\]
and 
\[
II_n := \E\left[\left(\int_t^T(H(s,X,\mu,\nu_s,Z^{\mu,\nu}_s) - H(s,X,\mu^n,\nu^n_s,Z^{\mu,\nu}_s))ds\right)^2 \right] \rightarrow 0,
\]
for all $t \in [0,T]$.

We first check that the integrands of $I_n$ and $II_n$ are uniformly integrable. Assumption (S.4) gives
\[
\left|g(X,\mu^n) - g(X,\mu)\right| \le c\left(2\psi(X) + \rho\left(\int \psi \, d\mu\right) + \rho\left(\int \psi \, d\mu_n\right)\right),
\]
which is indeed square integrable in light of (S.2) and (\ref{zconv1}). Note that
\begin{align}
|H(t,X,\mu,&\nu_t,Z^{\mu,\nu}_t) - H(t,X,\mu^n,\nu^n_t,Z^{\mu,\nu}_t)| \nonumber \\
	&\le \sup_{a \in A}\left|f(t,X,\mu,\nu_t,a) + Z^{\mu,\nu}_t \cdot \sigma^{-1}b(t,X,\mu,a) \right. \nonumber \\ 	
	& \left. \quad - f(t,X,\mu^n,\nu^n_t,a) - Z^{\mu,\nu}_t \cdot \sigma^{-1}b(t,X,\mu^n,a) \right| \nonumber \\
	&\le \left|\Delta^{f,n}_t\right| + |Z^{\mu,\nu}_t||\Delta^{b,n}_t|. \label{exist3}
\end{align}
where
\begin{align*}
\Delta^{f,n}_t &:= \sup_{a \in A}\left|f(t,X,\mu,\nu_t,a) - f(t,X,\mu^n,\nu^n_t,a)\right|, \text{  and} \\ 
\Delta^{b,n}_t &:= \sup_{a \in A}\left|\sigma^{-1}b(t,X,\mu,a) - \sigma^{-1}b(t,X,\mu^n,a)\right|.
\end{align*}
Again, (S.4) lets us bound $|\Delta^{f,n}|$ by the same term with which we bounded $|g(X,\mu^n) - g(X,\mu)|$. Since $Z^{\mu,\nu} \in \mathbb{H}^{2,1}$ and $|\Delta^{b,n}|$ is bounded, the integrands are indeed uniformly integrable.

It is clear now that $I_n \rightarrow 0$, because of assumption (E) and the dominated convergence theorem. Rewrite $II_n$ as
\begin{align*}
II_n = &\E\left[\left|\int_t^Tds\left(\int_{\P(A)}\nu_s(dq)H(s,X,\mu,q,Z^{\mu,\nu}_s) - \right.\right.\right. \\
&\quad\quad\quad\quad\quad\left.\left.\left. \int_{\P(A)}\nu^n_s(dq)H(s,X,\mu^n,q,Z^{\mu,\nu}_s)\right)\right|^2 \right].
\end{align*}
For fixed $s$ and $\omega$, the function $\Q \times \P(A) \ni (\mu^\prime, q) \mapsto H(s,X,\mu^\prime,q,Z^{\mu,\nu}_s)$ is continuous, by Lemma \ref{continuity}. Compactness of $\P(A)$ implies that the function $\Q \ni \mu^\prime \mapsto H(s,X,\mu^\prime,q,Z^{\mu,\nu}_s)$ is continuous, uniformly in $q$ (see Lemma \ref{lemma-joint-continuity}). Thus
\[
\int_{\P(A)}\nu^n_s(dq)H(s,X,\mu^n,q,Z^{\mu,\nu}_s) - \int_{\P(A)}\nu^n_s(dq)H(s,X,\mu,q,Z^{\mu,\nu}_s) \rightarrow 0.
\]
By definition of the stable topology of $\M$, we also have
\[
\int_t^Tds\int_{\P(A)}(\nu^n_s-\nu_s)(dq)H(s,X,\mu,q,Z^{\mu,\nu}_s) \rightarrow 0.
\]
It is now clear that $II_n \rightarrow 0$, and the proof is complete.
\end{proof}

The last ingredient of the proof is to establish the applicability of Proposition \ref{fixedpointtheorem}. Note that $A$ is a compact subset of a normed space, say $(A',\|\cdot\|_A)$, and thus $\A$ may also be viewed as a subset of the normed space of (equivalence classes of $dt \times dP$-a.e. equal) progressively measurable $A'$-valued processes, with the norm
\[
\|\alpha\|_\A := \E\int_0^T\|\alpha_t\|_Adt.
\]
\begin{lemma} \label{auhc}
Under assumptions (E) and (C), the function $\A : \Q \times \M \rightarrow 2^\A$ defined by (\ref{alphamudef}) is upper hemicontinuous and has closed and convex values.
\end{lemma}
\begin{proof}
Convexity follows immediately from assumption (C). We first show $\A(\cdot)$ has closed values. Let $\L$ denote Lebesgue measure on $[0,T]$. Note that $\|\cdot\|_A$ is bounded on $A$, and thus $\|\cdot\|_\A$ metrizes convergence in $\L \times P$-measure. To prove closedness, fix a sequence $\alpha^n \in \A(\mu,\nu)$ such that $\|\alpha^n - \alpha\|_\A \rightarrow 0$ for some $\alpha \in \A$. By passing to a subsequence, we may assume $\alpha^n_t(\omega) \rightarrow \alpha_t(\omega)$ for all $(t,\omega) \in N$, for some $N \subset [0,T] \times \Omega$ with $\L \times P(N) = 1$. We may assume also that $\alpha^n_t(\omega) \in A(t,X(\omega),\mu,Z^{\mu,\nu}_t(\omega))$ for all $n$ and $(t,\omega) \in N$. By Theorem \ref{maximizer}, for each $(t,\omega)$ the set $A(t,X(\omega),\mu,Z^{\mu,\nu}_t(\omega)) \subset A$ is compact, and thus $\alpha_t(\omega) \in A(t,X(\omega),\mu,Z^{\mu,\nu}_t(\omega))$ for all $(t,\omega) \in N$.

To prove upper hemicontinuity, let $(\mu^n,\nu^n) \rightarrow (\mu,\nu)$ in $\Q \times \M$. We must show that
\begin{align*}
d(\A(\mu,\nu), \A(\mu^n,\nu^n)) = \sup_{\alpha^n \in \A(\mu^n,\nu^n)}\inf_{\alpha \in \A(\mu,\nu)}\E\int_0^T\|\alpha^n_t - \alpha_t\|_Adt \rightarrow 0.
\end{align*}
Define 
\begin{align*}
A_t(\omega) &:= A(t,X(\omega),\mu,Z^{\mu,\nu}_t(\omega)), \\
A^n_t(\omega) &:= A(t,X(\omega),\mu^n,Z^{\mu^n,\nu^n}_t(\omega)),
\end{align*}
and
\begin{align*}
c^n_t(\omega) &:= d\left(A_t(\omega), A^n_t(\omega)\right) \\
	&= \sup\left\{\inf\left\{\|a - b\|_A : b \in A_t(\omega)\right\} : a \in A^n_t(\omega) \right\}.
\end{align*}
Lemma \ref{zconvergence} implies that $Z^{\mu^n,\nu^n} \rightarrow Z^{\mu,\nu}$ in $\L \times P$-measure; it follows then from upper hemicontinuity of $A(t,x,\cdot,\cdot)$ (Lemma \ref{continuity}) that $c^n \rightarrow 0$ in $\L \times P$-measure as well. Since of course $c^n$ is bounded, the proof will be complete once we establish
\[
\sup_{\alpha^n \in \A(\mu^n,\nu^n)}\inf_{\alpha \in \A(\mu,\nu)}\E\int_0^T\|\alpha^n_t - \alpha_t\|_Adt = \E\int_0^Tc^n_tdt.
\]

To prove that we can pass the infimum and supremum inside of the integrals, we first use Theorem 18.19 of \cite{aliprantisborder} to draw a number of conclusions. First, the map $(t,\omega) \mapsto A_t(\omega)$ is measurable, in the sense of Definition 18.1 of \cite{aliprantisborder}, and thus also weakly measurable since it is compact-valued (see Lemma 18.2 of \cite{aliprantisborder}). Second, there exists a measurable function $\hat{\beta} : [0,T] \times \Omega \times A \rightarrow A$ such that
\begin{align*}
\|a - \hat{\beta}(t,\omega,a)\|_A &= \inf\left\{\|a - b\|_A : b \in A_t(\omega)\right\}, \\
\hat{\beta}(t,\omega,a) &\in A_t(\omega).
\end{align*}
Note that for any $\alpha^n \in \A$, the process $\hat{\beta}(t,\omega,\alpha^n_t(\omega))$ is in $\A(\mu,\nu)$. Hence, we may exchange the infimum and the expectation to get
\[
\inf_{\alpha \in \A(\mu,\nu)}\E\int_0^T\|\alpha^n_t - \alpha_t\|_Adt = \E\int_0^T\inf\left\{\|\alpha^n_t - b\|_A : b \in A_t(\omega)\right\}dt,
\]
It follows from Theorem \ref{maximizer} that $a \mapsto \inf\left\{\|a - b\|_A : b \in A_t(\omega)\right\}$ is continuous for each $(t,\omega)$. Hence, Theorem 18.19 of \cite{aliprantisborder} also tells us that there exists a measurable selection $\hat{\beta}^n : [0,T] \times \Omega \rightarrow A$ such that
\begin{align*}
c^n_t(\omega) &= \inf\left\{\|\hat{\beta}^n(t,\omega) - b\|_A : b \in A_t(\omega)\right\}, \\
\hat{\beta}^n(t,\omega) &\in A^n_t(\omega).
\end{align*}
The process $\hat{\beta}^n(t,\omega)$ is in $\A(\mu^n,\nu^n)$, and so we exchange the supremum and the expectation to get
\[
\sup_{\alpha^n \in \A(\mu^n,\nu^n)}\E\int_0^T\inf\left\{\|\alpha^n_t - b\|_A : b \in A_t(\omega)\right\}dt = \E\int_0^Tc^n_tdt.
\]
\end{proof}

\begin{proof}[Proof of Theorem \ref{existence}]
The proof of Theorem \ref{existence} is an application of Proposition \ref{fixedpointtheorem}, with $K = \Q \times \M$ and $E = \A$. Let $\S$ denote the vector space of bounded measurable functions $\phi : [0,T] \times \P(A) \rightarrow \R$ such that $\phi(t,\cdot)$ is continuous for each $t$. Endow $\S$ with the supremum norm, and let $\S^*$ denote its continuous dual space. Note that $\M \subset \S^*$. Let $\Y := B_\psi(\C) \oplus \S$, endowed with the norm
\[
\|(\phi,\eta)\|_\Y := \sup_{x \in \C}\frac{|\phi(x)|}{\psi(x)} + \sup_{(t,q) \in [0,T] \times \P(A)}|\eta(t,q)|.
\]
The dual of $\Y$ is $\Y^* = B^*_\psi(\C) \oplus \S^*$, which contains $\Q \times \M$ as a subset. Using $\tau_\psi(\C)$ on $\Q$, the product topology of $\Q \times \M$ coincides with the topology induced by the weak*-topology of $\Y^*$. By Lemma \ref{RNDbound}, the function $\Phi$ takes values in $\Q \times \M$, noting that $\Q \times \M$ is convex and compact by Lemma \ref{Qcompact}. Let $\tau_\M$ denote the topology of $\M$. To prove that $\Phi : (\Q,\tau_\psi(\C)) \times (\A,\|\cdot\|_\A) \rightarrow (\Q,\tau_\psi(\C)) \times (\M,\tau_\M)$ is continuous, Lemma \ref{Qcompact} tells us that it suffices to show that $\Phi : (\Q,\tau_\psi(\C)) \times (\A,\|\cdot\|_\A) \rightarrow (\Q,\tau_1(\C)) \times (\M,\tau_\M)$ is sequentially continuous. We will instead prove the stronger statement that $\Phi : (\Q,\tau_\psi(\C)) \times (\A,\|\cdot\|_\A) \rightarrow (\Q,\V_1) \times (\M,\tau_\M)$ is sequentially continuous, where $\V_1$ denotes the total variation metric,
\[
\V_1(\mu,\nu) := \sup\int\phi\,d(\mu-\nu),
\]
where the supremum is over measurable real-valued functions $\phi$ with $|\phi| \le 1$. Denote by $\H(\nu | \mu)$ the relative entropy,
\[
\H(\nu|\mu) = \begin{cases}
\int \log \frac{d\nu}{d\mu} d\nu & \text{if } \nu \ll \mu \\
+ \infty & \text{otherwise}.
\end{cases}
\]

Now let $(\mu^n,\alpha^n) \rightarrow (\mu,\alpha)$ in $(\Q,\tau_\psi(\C)) \times (\A,\|\cdot\|_\A)$. We first show that $P^{\mu^n,\alpha^n} \rightarrow P^{\mu,\alpha}$. By Pinsker's inequality, it suffices to show
\[
\H(P^{\mu,\alpha}|P^{\mu^n,\alpha^n})) \rightarrow 0.
\]
Since
\[
\frac{dP^{\mu^n,\alpha^n}}{dP^{\mu,\alpha}} = \mathcal{E}\left(\int_0^\cdot\left(\sigma^{-1}b\left(t,X,\mu^n,\alpha^n_t\right) - \sigma^{-1}b\left(t,X,\mu,\alpha_t\right) \right) dW^{\mu,\alpha}_t \right)_T,
\]
and since $\sigma^{-1}b$ is bounded, we compute
\begin{align*}
\H(P^{\mu,\alpha}|P^{\mu^n,\alpha^n}) &= -\E^{\mu,\alpha}\left[\log\frac{dP^{\mu^n,\alpha^n}}{dP^{\mu,\alpha}} \right] \\
	&= \frac{1}{2}\E^{\mu,\alpha}\left[\int_0^T\left|\sigma^{-1}b\left(t,X,\mu^n,\alpha^n_t\right) - \sigma^{-1}b\left(t,X,\mu,\alpha_t\right) \right|^2dt\right]. 
\end{align*}
Since $P^{\mu,\alpha} \sim P$ and $\alpha^n \rightarrow \alpha$ in $\L \times P$-measure, it follows from Lemma \ref{zconvergence} that $Z^{\mu^n,\nu^n} \rightarrow Z^{\mu,\nu}$ in $\L \times P^{\mu,\alpha}$-measure, where $\L$ denotes Lebesgue measure on $[0,T]$. By assumption (E), the map $\sigma^{-1}b(t,x,\cdot,\cdot)$ is continuous for each $(t,x)$. Conclude from the bounded convergence theorem that $P^{\mu^n,\alpha^n} \rightarrow P^{\mu,\alpha}$ in total variation. It follows immediately that $P^{\mu^n,\alpha^n} \circ X^{-1} \rightarrow P^{\mu,\alpha} \circ X^{-1}$ in total variation, and that
\[
\V_1\left(P^{\mu^n,\alpha^n} \circ (\alpha^n_t)^{-1}, P^{\mu,\alpha} \circ (\alpha^n_t)^{-1}\right) \le \V_1\left(P^{\mu^n,\alpha^n}, P^{\mu,\alpha}\right) \rightarrow 0.
\]
Moreover, $P^{\mu,\alpha} \circ (\alpha^n_t)^{-1} \rightarrow P^{\mu,\alpha} \circ \alpha_t^{-1}$ in $\L$-measure, since $\alpha^n \rightarrow \alpha$ in $\L \times P$-measure. Thus $P^{\mu^n,\alpha^n} \circ (\alpha^n_t)^{-1} \rightarrow P^{\mu,\alpha} \circ \alpha_t^{-1}$ in $\L$-measure, which finally implies
\[
\delta_{P^{\mu^n,\alpha^n} \circ (\alpha^n_t)^{-1}}(dq)dt \rightarrow \delta_{P^{\mu,\alpha} \circ \alpha_t^{-1}}(dq)dt, \text{  in } \M.
\]

With continuity of $\Phi$ established, $\Phi$ and $\A(\cdot)$ verify the assumptions of Proposition \ref{fixedpointtheorem}, and thus there exists a fixed point $(\mu,\nu) \in \Phi(\mu,\A(\mu,\nu)) = \{\Phi(\mu,\alpha) : \alpha \in \A(\mu,\nu)\}$. It remains to notice that the function $\Phi$ takes values in $\Q \times \M^0$, where
\begin{align*}
\M^0 := &\left\{\nu \in \M : \nu(dt,dq) = \delta_{\hat{q}(t)}(dq)dt \right. \\
	&\quad\quad\quad\left. \text{ for some measurable } \hat{q} : [0,T] \rightarrow \P(A)\right\}.
\end{align*}
For an element in $\M^0$, the correponding map $\hat{q}$ is uniquely determined, up to almost everywhere equality. Hence, for our fixed point $(\mu,\nu)$, we know that there exist $\alpha \in \A(\mu,\nu)$ and a measurable function $\hat{q} : [0,T] \rightarrow \P(A)$ such that $\nu_t = \delta_{\hat{q}(t)}$ and $\hat{q}(t) = P^{\mu,\alpha} \circ \alpha_t^{-1}$ for almost every $t$.
\end{proof}

\begin{remark} \label{strongformulation}
Assume for the moment that there is no mean field interaction in the control. Following the notation of Remark \ref{zdifficulty}, we may ask if the SDE
\[
dX_t = b(t,X,\mu,\hat{\alpha}(t,X,\mu,\zeta_\mu(t,X)))dt + \sigma(t,X)dW_t,
\]
admits a strong solution, with $\mu$ equal to the law of $X$. This would allow us to solve the mean field game in a strong sense, on a given probability space, as is required in \cite{carmonadelarue-mfg} and \cite{bensoussan-lqmfg}. Since $\zeta_\mu(t,X) = Z^\mu_t$, this forward SDE is coupled with the backward SDE:
\[
\begin{cases}
dX_t \!\!\!\!\!\!&= b(t,X,\mu,\hat{\alpha}(t,X,\mu,Z_t))dt + \sigma(t,X)dW_t, \\
dY_t &= -H(t,X,\mu,Z_t)dt + Z_tdW_t, \\
\mu_0 &= \lambda_0, \quad X \sim \mu, \quad Y_T = g(X,\mu).
\end{cases}
\]
To solve the mean field game in a strong sense, one must therefore resolve this ``mean field FBSDE'', studied in some generality in \cite{carmonadelarue-mffbsde}. The solution must consist of $(X,Y,Z,\mu)$, such that $(X,Y,Z)$ are processes adapted to the filtration generated by $(W_t,X_0)_{t \in [0,T]}$ and satisfying the above SDEs, and such that the law of $X$ is $\mu$. Our formulation is a relaxation of the more common formulation (e.g. \cite{carmonadelarue-mfg} and \cite{bensoussan-lqmfg}) in that the forward SDEs no longer need to be solved in a strong sense. Note, however, that the FBSDE written here is of a different nature from those of \cite{carmonadelarue-mfg, bensoussan-lqmfg}, which were obtained from the maximum principle. Our FBSDE is more like a stochastic form of the PDE systems of Lasry and Lions; indeed, in the Markovian case, the Feynman-Kac formula for the backward part is nothing but the HJB equation.
\end{remark}

\subsection{Proof of Theorem \ref{uniqueness} (uniqueness)} 
\label{section-uniquenessproof}
\begin{proof}[Proof of Theorem \ref{uniqueness}]
Recall that $\A(\mu,\nu)$ is always nonempty, as in Remark \ref{nonempty}. By condition (U.1), we know $\A(\mu,\nu)$ is a singleton for each $(\mu,\nu) \in \P_\psi(\C) \times \M$. Its unique element $\alpha^{\mu,\nu}$ is defined given by
\[
\alpha^{\mu,\nu}_t = \hat{\alpha}(t,X,Z^{\mu,\nu}_t),
\]
where the function $\hat{\alpha}$ is defined as in (\ref{selection}); note that assumptions (U.2) and (U.3) imply that $\hat{\alpha} = \hat{\alpha}(t,x,z)$ does not depend on $\mu$ or $\nu$. Suppose now that $(\mu^1,\nu^1),(\mu^2,\nu^2) \in \P_\psi(\C) \times \M$ are two solutions of the MFG; that is, they are fixed points of the (single-valued) function $\Phi(\cdot,\A(\cdot))$. Abbreviate $Y^i = Y^{\mu^i,\nu^i}$, $Z^i = Z^{\mu^i,\nu^i}$, $\alpha^i = \alpha^{\mu^i,\nu^i}$, $f^i_t := f(t,X,\mu^i,\nu^i_t,\alpha^i_t)$ and $b^i_t := \sigma^{-1}b(t,X,\alpha^i_t)$. We begin by rewriting the BSDEs (\ref{BSDEvalue}) in two ways:
\begin{align*}
d(Y^1_t-Y^2_t) &= -\left[f^1_t - f^2_t + Z^1_t \cdot b^1_t - Z^2_t \cdot b^2_t\right]dt + (Z^1_t-Z^2_t)dW_t \\
	&= -\left[f^1_t - f^2_t + Z^2_t \cdot \left(b^1_t - b^2_t\right)\right]dt + (Z^1_t-Z^2_t)dW^{\mu^1,\alpha^1}_t \\
	&= -\left[f^1_t - f^2_t + Z^1_t \cdot \left(b^1_t - b^2_t\right)\right]dt + (Z^1_t-Z^2_t)dW^{\mu^2,\alpha^2}_t,
\end{align*}
with $Y^1_T - Y^2_T = g(X,\mu^1) - g(X,\mu^2)$. Recall that $P^{\mu,\alpha}$ agrees with $P$ on $\F_0$ for each $\mu \in \P_\psi(\C)$ and $\alpha \in \A$. In particular,
\[
\E^{\mu^1,\alpha^1}\left[Y^1_0-Y^2_0\right] = \E\left[Y^1_0-Y^2_0\right] = \E^{\mu^2,\alpha^2}\left[Y^1_0-Y^2_0\right].
\]
Thus, if $\Delta g(X) := g(X,\mu^1) - g(X,\mu^2)$, then
\begin{align}
\E\left[Y^1_0-Y^2_0\right] &= \E^{\mu^1,\alpha^1}\left[\Delta g(X) + \int_0^T\left(f^1_t - f^2_t + Z^2_t \cdot \left(b^1_t - b^2_t\right)\right)dt\right] \label{unique1} \\
	&= \E^{\mu^2,\alpha^2}\left[\Delta g(X) + \int_0^T\left(f^1_t - f^2_t + Z^1_t \cdot \left(b^1_t - b^2_t\right)\right)dt\right]. \label{unique2}
\end{align}

Since the optimal control maximizes the Hamiltonian,
\begin{align*}
f^1_t + Z^2_t \cdot b^1_t &= h(t,X,\mu^1,\nu^1_t,Z^2_t,\alpha^1_t) \le H(t,X,\mu^1,\nu^1_t,Z^2_t) \\
		&= f_1(t,X,\mu^1) + f_2(t,\mu^1,\nu^1_t) + f_3(t,X,\alpha^2_t) + Z^2_t \cdot b^2_t.
\end{align*}
Thus, since
\[
f^2_t = f_1(t,X,\mu^2) + f_2(t,\mu^2,\nu^2_t) + f_3(t,X,\alpha^2_t),
\]
defining $\Delta f_1(t,X) := f_1(t,X,\mu^1) - f_1(t,X,\mu^2)$ yields
\begin{align}
f^1_t - f^2_t + Z^2_t \cdot \left(b^1_t - b^2_t\right) \le \Delta f_1(t,X) + f_2(t,\mu^1,\nu^1_t) - f_2(t,\mu^2,\nu^2_t). \label{unique3}
\end{align}
By switching the place of the indices, the same argument yields
\begin{align}
f^1_t - f^2_t + Z^1_t \cdot \left(b^1_t - b^2_t\right) \ge \Delta f_1(t,X) + f_2(t,\mu^1,\nu^1_t) - f_2(t,\mu^2,\nu^2_t). \label{unique4}
\end{align}
Since $f_2(t,\mu^i,\nu^i_t)$ are deterministic, applying inequality (\ref{unique3}) to (\ref{unique1}) and (\ref{unique4}) to (\ref{unique2}) yields
\[
0 \le \left[\E^{\mu^1,\alpha^1} - \E^{\mu^2,\alpha^2}\right]\left[\Delta g(X) + \int_0^T\Delta f_1(t,X)dt\right].
\]
Hypothesis (U.4) implies that the right side is at most zero, so in fact
\begin{align}
0 = \left[\E^{\mu^1,\alpha^1} - \E^{\mu^2,\alpha^2}\right]\left[\Delta g(X) + \int_0^T\Delta f_1(t,X)dt\right]. \label{unique5}
\end{align}

Suppose $\alpha^1 \neq \alpha^2$ holds on a $(t,\omega)$-set of strictly positive $\L \times P$-measure, where $\L$ is again Lebesgue measure. Then assumption (U.1) implies that the inequalities (\ref{unique3}) and (\ref{unique4}) are strict on a set of positive $\L \times P$-measure. Since $P \sim P^{\mu^1,\alpha^1} \sim P^{\mu^2,\alpha^2}$, this implies
\[
0 < \left[\E^{\mu^1,\alpha^1} - \E^{\mu^2,\alpha^2}\right]\left[\Delta g(X) + \int_0^T\Delta f_1(t,X)dt\right],
\]
which contradicts (\ref{unique5}). Thus $\alpha^1 \neq \alpha^2$ must hold $\L \times P$-$a.e.$, which yields
\begin{align*}
\frac{dP^{\mu^1,\alpha^1}}{dP} &= \mathcal{E}\left(\int_0^\cdot\sigma^{-1}b(t,X,\alpha^1_t) dW_t\right)_T \\
	&= \mathcal{E}\left(\int_0^\cdot\sigma^{-1}b(t,X,\alpha^2_t) dW_t\right)_T = \frac{dP^{\mu^2,\alpha^2}}{dP}, \ \ a.s.
\end{align*}
Thus $\mu^1 = P^{\mu^1,\alpha^1} \circ X^{-1} = P^{\mu^2,\alpha^2} \circ X^{-1} = \mu^2$, and $\nu^1_t = \delta_{P^{\mu^1,\alpha^1} \circ (\alpha^1_t)^{-1}} = \delta_{P^{\mu^2,\alpha^2} \circ (\alpha^2_t)^{-1}} = \nu^2_t$ a.e. 
\end{proof}

\section{Proof of finite-player approximation theorems} \label{section-finiteproof}
This section justifies the mean field approximation by proving Theorem \ref{approximationtheorem}, the general approximation result, as well as Proposition \ref{pr:priceimpact}, the rate of convergence for the price impact model.

\subsection{Proof of Theorem \ref{approximationtheorem}}
We work on the probability space of Section \ref{section-finite}. Recall that under $P$, $X^1,X^2,\ldots$ are i.i.d. with common law $\hat{\mu}$ and $\alpha^1_t,\alpha^2_t,\ldots$ are i.i.d. with common law $\hat{q}_t$, for almost every $t$. By symmetry, we may prove the result for player 1 only. For $\beta \in \A_n$, define $\beta^\alpha := (\beta,\alpha^2,\ldots,\alpha^n) \in \A^n_n$. We abuse notation somewhat by writing $\alpha$ in place of $(\alpha^1,\ldots,\alpha^n) \in \A^n_n$. Note that $(\alpha^1)^\alpha = \alpha$ and $P_n(\alpha) = P$, in our notation. For $\beta \in \A_n$, let
\begin{align*}
J'_n(\beta) &:= \E^{P_n(\beta^\alpha)}\left[\int_0^Tf(t,X^1,\hat{\mu},\hat{q}_t,\beta_t)dt + g(X^1,\hat{\mu}) \right].
\end{align*}
Note that $J'_n(\alpha^1)$ does not depend on $n$. We divide the proof into three lemmas.

\begin{lemma} \label{empirical2}
Let $F : \C \times \P_\psi(\C) \rightarrow \R$ be empirically measurable, and suppose $F(x,\cdot)$ is $\tau_\psi(\C)$ continuous at $\hat{\mu}$ for each $x \in \C$. Assume also that there exists $c > 0$ such that
\[
|F(x,\mu)| \le c\left(\psi(x) + \int\psi\,d\mu\right), \text{ for all } (x,\mu) \in \C \times \P_\psi(\C).
\]
Then $\lim_{n \rightarrow \infty}\E[|F(X^i,\mu^n) - F(X^i,\hat{\mu})|^p] = 0$ for each $i$ and $p \in [1,2)$.
\end{lemma}
\begin{proof}
By symmetry, it suffices to prove this for $i = 1$. By replacing $F(x,\mu)$ with $|F(x,\mu) - F(x,\hat{\mu})|$, assume without loss of generality that $F \ge 0$ and $F(x,\hat{\mu}) = 0$ for all $x$. Define
\[
\nu^n := \frac{1}{n-1}\sum_{i=2}^n\delta_{X^i}.
\]
By independence of $X^1$ and $\nu^n$, we have
\begin{align*}
\E[F(X^1,\mu^n)] &= \E\left[\E\left[F\left(x, \frac{1}{n}\delta_x + \frac{n-1}{n}\nu^n\right)\right]_{x = X^1}\right].
\end{align*}
Now let $\epsilon > 0$. By continuity of $F(x,\cdot)$, there exist $\delta > 0$ and $\phi_1,\ldots,\phi_k \in B_\psi(\C)$ such that $F(x,\nu) < \epsilon$ whenever $|\int\phi_id(\nu - \hat{\mu})| < \delta$ for all $i=1,\ldots,k$. By the law of large numbers,
\[
\lim_{n\rightarrow\infty}\left|\int\phi_i\,d\left(\frac{1}{n}\delta_x + \frac{n-1}{n}\nu^n - \hat{\mu}\right)\right| = 0, \ \ a.s.
\]
Thus
\[
\limsup_{n\rightarrow\infty}F\left(x, \frac{1}{n}\delta_x + \frac{n-1}{n}\nu^n\right) \le \epsilon, \ a.s.,
\]
for each $\epsilon > 0$, and so $F\left(x, \frac{1}{n}\delta_x + \frac{n-1}{n}\nu^n\right) \rightarrow 0$ a.s. for each $x$. The growth assumption along with (S.2) yield
\[
\E\left[F^2(X^1,\mu^n)\right] \le 2c^2\E\left[\psi^2(X^1) + \left(\int\psi\,d\mu^n\right)^2\right] \le 4c^2\E\left[\psi^2(X^1)\right] < \infty,
\]
and we conclude by the dominated convergence theorem.
\end{proof}

\begin{lemma} \label{step1}
We have $\lim_{n\rightarrow \infty}\sup_{\beta \in \A_n}|J_{n,1}(\beta^\alpha) - J'_n(\beta)| = 0.$
\end{lemma}
\begin{proof}
Note that, for any $\beta \in \A_n$,
\begin{align}
|J_{n,1}(\beta^\alpha) - J'_n(\beta)| 
	&\le \int_0^T\E^{P_n(\beta^\alpha)}[F_t(X^1,\mu^n) + G_t(X^1,q^n(\beta^\alpha_t))]dt  \nonumber \\
	&\quad\quad\quad + \E^{P_n(\beta^\alpha)}[|g(X^1,\mu^n)-g(X^1,\hat{\mu})|], \label{inequality-convergence}
\end{align}
where $F : [0,T] \times \C \times \P_\psi(\C) \rightarrow \R$ and $G : [0,T] \times \C \times \P(A) \rightarrow \R$ are defined by
\begin{align*}
F_t(x,\mu) &:= \sup_{(a,q) \in A \times \P(A)}|f(t,x,\mu,q,a) - f(t,x,\hat{\mu},q,a)|, \\
G_t(x,q) &:= \sup_{a \in A}|f(t,x,\hat{\mu},q,a) - f(t,x,\hat{\mu},\hat{q}_t,a)|.
\end{align*}
Theorem 18.19 of \cite{aliprantisborder} ensures that both functions are (empirically) measurable. Since $A$ and $\P(A)$ are compact, Lemma \ref{lemma-joint-continuity} assures us that $F_t(x,\cdot)$ is $\tau_\psi(\C)$-continuous at $\hat{\mu}$ and that $G_t(x,\cdot)$ is weakly continuous, for each $(t,x)$. Similar to the proof of Lemma \ref{RNDbound}, $\{dP_n(\beta^\alpha)/dP : \beta \in \A_n, \ n \ge 1\}$ are bounded in $L^p(P)$, for any $p \ge 1$. Since assumption (F.5) is uniform in $t$ for $f$, we deduce from Lemma \ref{empirical2} and the dominated convergence theorem that
\[
\lim_{n \rightarrow \infty}\sup_{\beta \in \A_n}\left[\int_0^T\E^{P_n(\beta^\alpha)}[F_t(X^1,\mu^n)]dt + \E^{P_n(\beta^\alpha)}[|g(X^1,\mu^n)-g(X^1,\hat{\mu})|]\right] = 0.
\]

It remains to check that the $G_t$ term converges. Note that $G_t(x,\cdot)$ is uniformly continuous, as $\P(A)$ is compact. Also $\V_1(q^n(\beta^\alpha_t), q^n(\alpha_t)) \le 2/n$, since these are empirical measures of $n$ points which differ in only one point (recall that $\V_1$ denotes total variation). Hence
\[
\lim_{n \rightarrow \infty}\sup_{\beta \in \A_n}\left|G_t(X^1,q^n(\alpha_t)) - G_t(X^1,q^n(\beta^\alpha_t))\right| = 0, \ \ a.s.
\]
Since $\alpha^1_t,\alpha^2_t,\ldots$ are i.i.d. with common law $\hat{q}_t$, we have $q^n(\alpha_t) \rightarrow \hat{q}_t$ weakly a.s. (see \cite{varadarajan-convergence}), and thus $G_t(X^1,q^n(\alpha_t)) \rightarrow 0$ a.s. Note that $dP_n(\beta^\alpha)/dP$ are bounded in $L^p(P)$ for any $p \ge 1$ and that the integrands above are bounded in $L^p(P)$ for any $p \in [1,2)$, by (F.5) and the same argument as in the proof of Lemma \ref{step1}. The dominated convergence theorem completes the proof.
\end{proof}

\begin{lemma} \label{step2}
For any $\beta \in \A_n$, $J'_n(\alpha^1) \ge J'_n(\beta)$.
\end{lemma}
\begin{proof}
We use the comparison principle for BSDEs. Fix $n$ and $\beta \in \A_n$. Define $\phi, \ \widetilde{\phi} : [0,T] \times \Omega \times \R^d \rightarrow \R$ by
\begin{align*}
\phi(t,z) &:= \sup_{a \in A}\left[f(t,X^1,\hat{\mu},\hat{q}_t,a) + z \cdot \left(\sigma^{-1}b(t,X^1,a) - \sigma^{-1}b(t,X^1,\alpha^1_t)\right)\right] \\
\widetilde{\phi}(t,z) &:= f(t,X^1,\hat{\mu},\hat{q}_t,\beta_t) + z \cdot \left(\sigma^{-1}b(t,X^1,\beta_t) - \sigma^{-1}b(t,X^1,\alpha^1_t)\right)
\end{align*}
By Pardoux and Peng \cite{pardouxpengbsde}, there exist unique solutions $(Y,Z^1,\ldots,Z^n)$ and $(\widetilde{Y},\widetilde{Z}^1,\ldots,\widetilde{Z}^n)$ of the BSDEs
\[
\begin{cases}
dY_t \!\!\!\!\!\!&= -\phi_n(t,Z^1_t)dt + \sum_{j=1}^nZ^j_t dW^j_t \\
Y_T &= g(X^1,\hat{\mu}),
\end{cases}
\]
\[
\begin{cases}
d\widetilde{Y}_t \!\!\!\!\!\!&= -\widetilde{\phi}_n(t,\widetilde{Z}^1_t)dt + \sum_{j=1}^n\widetilde{Z}^j_t dW^j_t \\
\widetilde{Y}_T &= g(X^1,\hat{\mu}).
\end{cases}
\]
The unique solution of the first BSDE is in fact given by $Z^2 \equiv \ldots \equiv Z^n \equiv 0$, where $(Y,Z^1)$ are $\mathbb{X}^1$-progressively measurable and solve the BSDE
\[
\begin{cases}
dY_t \!\!\!\!\!\!&= -\left[H(t,X^1,\hat{\mu},\hat{q}_t,Z^1_t) - Z^1_t \cdot \sigma^{-1}b(t,X^1,\alpha^1_t)\right]dt + Z^1_t dW^1_t \\
Y_T &= g(X^1,\hat{\mu}).
\end{cases}
\]
This is due to the $\mathbb{X}^1$-measurability of the driver and terminal condition of this BSDE. Recall that $\alpha^1$ is optimal for the mean field problem, and thus it must maximize the Hamiltonian; that is,
\begin{align*}
H(t,X^1,\hat{\mu},\hat{q}_t,Z^1_t) &= h(t,X^1,\hat{\mu},\hat{q}_t,Z^1_t,\alpha^1_t) \\
	&= f(t,X^1,\hat{\mu},\hat{q}_t,\alpha^1_t) + Z^1_t \cdot \sigma^{-1}b(t,X^1,\alpha^1_t).
\end{align*}
Thus $dY_t = -f(t,X^1,\hat{\mu},\hat{q}_t,\alpha^1_t)dt + Z^1_t dW^1_t$. Since $W^1$ is a Wiener process under $P$, taking expectations yields $\E[Y_0] = J'_n(\alpha^1)$, which we note does not depend on $n$.

Similarly, note that $W^j$, $j \ge 2$ are Wiener processes under $P_n(\beta^\alpha)$, as is $W^{\beta,1}$. Hence, we rewrite $\widetilde{Y}$ as follows:
\[
\begin{cases}
d\widetilde{Y}_t \!\!\!\!\!\!&= -f(t,X^1,\hat{\mu},\hat{q}_t,\beta_t)dt + \widetilde{Z}^1_t dW^{\beta,1}_t + \sum_{j=2}^n\widetilde{Z}^j_t dW^j_t \\
\widetilde{Y}_T &= g(X^1,\hat{\mu}).
\end{cases}
\]
Take expectations, noting that $P = P_n(\beta^\alpha)$ on $\F^n_0$, to see $\E[\widetilde{Y}_0] = \E^{P_n(\beta^\alpha)}[\widetilde{Y}_0] = J'_n(\beta)$. Finally, since $\phi \ge \widetilde{\phi}$, the comparison principle for BSDEs yields $Y_0 \ge \widetilde{Y}_0$, and thus $J'_n(\beta) \le J'_n(\alpha^1)$.
\end{proof}

\begin{proof}[Proof of Theorem \ref{approximationtheorem}]
Simply let $\epsilon_n = 2\sup_{\beta \in \A_n}|J_{n,1}(\beta^\alpha) - J'_n(\beta)|$. Then $\epsilon_n \rightarrow 0$ by Lemma \ref{step1}, and Lemma \ref{step2} yields, for all $\beta \in \A_n$,
\[
J_{n,1}(\beta^\alpha) \le \frac{1}{2}\epsilon_n + J'_n(\beta) \le \frac{1}{2}\epsilon_n + J'_n(\alpha^1) \le \epsilon_n + J_{n,1}(\alpha).
\]
\end{proof}

\subsection{Proof of Proposition \ref{pr:priceimpact}}
\begin{proof}[Proof of Proposition \ref{pr:priceimpact}]
We simply modify the proof of Theorem \ref{approximationtheorem}, in light of the special structure of the price impact model. Namely, the inequality (\ref{inequality-convergence}) becomes
\begin{align*}
\epsilon_n &= 2\sup_{\beta \in \A_n}|J_{n,1}(\beta^\alpha) - J'_n(\beta)| \\
	&\le 2\sup_{\beta \in \A_n}\E^{P_n(\beta)}\int_0^T\left|\gamma X^1_t\int_A c'd(q^n_t(\beta^\alpha) - \hat{q}_t)\right|dt.
\end{align*}
Use H\"{o}lder's inequality to get
\begin{align*}
\epsilon_n &\le 2\gamma\E\left[\|X^1\|^4\right]^{1/4}\sup_{\beta \in \A_n}\E\left[\left(\frac{dP_n(\beta)}{dP}\right)^4\right]^{1/4}\int_0^TF_t(\beta)^{1/2}dt,
\end{align*}
where
\[
F_t(\beta) := \E\left[\left(\int_A c'd(q^n_t(\beta^\alpha) - \hat{q}_t)\right)^2\right]
\]
Assumption (S.2) with $\psi(x) = e^{c_1\|x\|}$ implies that $\|X^1\|$ has finite moments of all orders. Again, $\{dP_n(\beta^\alpha)/dP : \beta \in \A_n, \ n \ge 1\}$ are bounded in $L^p(P)$ for any $p \ge 1$. So it suffices to show
\[
\sup_{\beta \in \A_n}F_t(\beta) \le C/n,
\]
for some $C > 0$. This will follow from two inequalities: An easy calculation gives 
\[
\left|\int_A c'd(q^n_t(\beta^\alpha) - \hat{q}_t)\right| \le 2C_1/n + \left|\int_A c'd(q^n_t(\alpha) - \hat{q}_t)\right|,
\]
where $C_1 = \sup_{a \in A}|c'(a)|$. Since $\alpha^1_t,\alpha^2_t,\ldots$ are i.i.d. with common law $\hat{q}_t$, 
\begin{align*}
\E\left[\left(\int_A c'd(q^n_t(\alpha) - \hat{q}_t)\right)^2\right] = \mathrm{Var}(c'(\alpha^1_t))/n \le 4C^2_1/n.
\end{align*}
\end{proof}

\section{Conclusions}
This paper provides a theoretical framework for fairly general mean field games with uncontrolled volatility, allowing us to prove new existence and uniqueness results for several types of mean field interactions which arise naturally in applications, and which were not studied before. Such results include models with rank effects, nearest-neighbor (e.g. quantile) interactions, and mean field interactions through the controls. The strength of our approach is its generality; the existence, uniqueness, and approximation results apply easily to many concrete models. More refined analysis, for example of regularity of solutions or numerics, could in theory be based on the McKean-Vlasov FBSDE discussed in Remark \ref{strongformulation}, which essentially provides a probabilistic representation of the PDE approach of \cite{lasrylionsmfg}. This is left for further investigation. 

\bibliographystyle{amsplain}
\bibliography{mfgweakbib}

\end{document}